%% file: acoustic-metamaterial-revision-07-10-2020-fixed.tex
\documentclass[10pt]{amsart}

\usepackage{amsmath, amsthm, amssymb, amsfonts, verbatim}

\usepackage[active]{srcltx}
\usepackage{array}
\usepackage[bookmarks]{hyperref}
\usepackage{amscd}
\usepackage{epsfig}
\usepackage{color}
\usepackage{comment}
\usepackage{setspace}
\usepackage{textcomp}
\usepackage{url}
\usepackage{multirow}
\usepackage{stmaryrd}
\usepackage{pdfsync}
\usepackage{breqn}

\numberwithin{equation}{section}

\input{./jlmacros.tex}

\newcommand{\fb}{\bs{f}}

\newcommand{\ub}{\bs{u}}

\newcommand{\vb}{\bs{v}}
\newcommand{\Vb}{\bs{V}}

\newcommand{\Wb}{\bs{W}}

\newcommand{\wb}{\bs{w}}

\newtheorem{definition}[theorem]{Definition}

\begin{document}

\title[Mixed method for metamaterials ]{A mixed method for time-transient acoustic wave propagation in metamaterials}

\author{Jeonghun J. Lee} 
\address{Department of Mathematics, Baylor University, Waco, TX , USA}
\email{jeonghun\_lee@baylor.edu}
\urladdr{}
\date{June, 2, 2020}

\maketitle

\begin{abstract}
In this paper we develop a finite element method for acoustic wave propagation in Drude-type metamaterials. 
The governing equation is written as a symmetrizable hyperbolic system with auxiliary variables. 
The standard mixed finite elements and discontinuous finite elements are used for spatial discretization, and the Crank--Nicolson scheme is used for time discretization. The a priori error analysis of fully discrete scheme is carried out in details. 
Numerical experiments illustrating the theoretical results and metamaterial wave propagation, are included.
\end{abstract}

\section{Introduction}
Metamaterials usually mean the materials with artificial micro/nano-scale structures which show unconventional macro-scale material properties which are not observed in natural materials. The unconventional material properties of metamaterials have many potential applications in wave propagation. For example, cloaking devices, which hide internal objects from external detection using wave refection, can be made by an appropriate design of metamaterial device. Therefore devising metamaterials and its numerical simulations are research topics of great interest nowadays.

There are three major classes of metamaterials, which are for acoustic, electromagnetic, and elastodynamic wave propagation. 
In this paper we only consider acoustic wave propagation in metamaterials. In time-harmonic cases some of these wave propagation equations coincide under special circumstances but they are all different in time transient wave propagation. For the theory of electromagnetic metamaterials and time-domain finite element methods we refer to, e.g., \cite{Hunag-Li-Yang:2013,Li-book-metamaterial:2013,Li-Wood:2007,Yang-Huang-Li:2016} and the references in \cite{Li:2016} for more comprehensive list of previous studies. There are also previous studies on elastodynamic metamaterials in, e.g., \cite{Milton-Seppecher:2008,Milton-Willis:2007,Norris-Shuvalov:2011}. 

To the best of our knowledge there are very limited number of previous studies on numerical methods for time transient acoustic wave propagation in metamaterials. In \cite{Bellis-Lombard:2019} some acoustic metamaterial models, the acoustic counterpart of doubly negative index materials in electromagnetics \cite{Li-Wood:2007,Yang-Huang-Li:2016}, are studied. In the paper the authors proposed a form of symmetrizable hyperbolic system as the governing equations of acoustic wave propagation in metamaterials. 
In addition, they proved existence of weak solutions and showed numerical experiments with the finite difference method. 

In this paper we develop a finite element method for the system proposed in \cite{Bellis-Lombard:2019} and prove the a priori error analysis. For spatial discretization we use the mixed finite element for the Poisson equation and some discontinuous finite element spaces. To circumvent lower convergence rate of the pressure variable in some mixed finite element pairs, we propose a novel local post-processing which gives numerical pressure with higher order approximation properties (See Subsection~3.3).

The paper is organized as follows. In Section~\ref{sec:prelim} we first introduce symbols and notation in the paper, and then present the governing equations for the acoustic wave propagation in metamaterials as well as the energy estimate. In Section~\ref{sec:error-analysis} we introduce finite element discretization for the system and prove the a priori error analysis. In particular, we show that a local post-processing for the pressure variable can be used to obtain a numerical pressure which has better approximation property than the original numerical pressure. In Section~\ref{sec:numerical} we present the results of numerical experiments which illustrate our theoretical results and exotic wave propagation in metamaterials.

\section{Preliminaries}\label{sec:prelim}

\subsection{Notation} \label{subsec:notation}

Let $\Omega \subset \mathbb{R}^d$, $d =2$ or $3$, be a polygonal/polyhedral domain with Lipschitz boundary. 
Throughout this paper we assume that $\mathcal{T}_h$ is a triangulation of $\Omega$ without hanging nodes.

We use $L^r(\Omega)$ to denote the Lebesgue space with the norm  
\algns{
  \| v \|_{L^r} =
  \begin{cases}
    \left( \int_{\Omega} |v|^r \, dx \right)^{1/r}, & \text{ if }1 \leq r < \infty, \\
    \esssup_{x \in \Omega} \{ |v(x)| \}, & \text{ if } r = \infty.
  \end{cases}
}
For a domain $D \subset \Omega$, $L^2(D)$ and $L^2(D;\R^d)$ be the sets of $\R$- and $\R^d$-valued square integrable functions with inner products $\LRp{v, v'}_D := \int_{D} v v' \,d x$ and
$\LRp{\vb, \vb'}_D := \int_{D} \vb \cdot \vb' \,d x$. We will use $\LRp{\cdot, \cdot}$ 
instead of $\LRp{\cdot, \cdot}_D$ if $D = \Omega$. For an integer $l \ge 0$ $\mc{P}_l(D)$ and $\mc{P}_l(D;\R^d)$
are the spaces of $\R$- and $\R^d$-valued polynomials of degree $\le l$ on $D$. 

In the paper $H^s(D)$, $s \ge 0$, denotes the Sobolev space based
on the $L^2$-norm with $s$-differentiability on the domain $D$. We refer to \cite{Evans-book} for a rigorous definition of this
space. The norm on $H^s(D)$ is denoted by $\| \cdot \|_{s,D}$ and $D$ is omitted if $D = \Omega$. 
If $\rho$ is a nonnegative function in $L^\infty(\Omega)$, then $\| v \|_{\rho}$ and $\| \vb \|_{\rho}$ denotes the $\rho$-weighted $L^2$-norms $\LRp{\int_{\Omega} \rho |v|^2 \,dx}^{1/2}$ and $\LRp{\int_{\Omega} \rho \vb \cdot \vb \,dx}^{1/2}$.

For $T >0$ and a separable Hilbert space $X$, let $C^0 ([0, T] ; X)$ denote the set
of functions $f : [0, T] \rightarrow X$ that are continuous
in $t \in [0, T]$. For an integer $m \geq 1$, we define
\begin{equation*}
  C^m ([0, T]; {X}) = \{ f \, | \, \partial^i f/\partial t^i \in C^0([0, T];X), \, 0 \leq i \leq m \},
\end{equation*}
where $\partial^i f/\partial t^i$ is the $i$-th time derivative in the
sense of the Fr\'echet derivative in ${X}$ (cf.~\cite{Yosida-book}).
For a function $f : [0, T] \rightarrow {X}$, the Bochner norm is defined as 
\algns{
  \| f \|_{L^r(0, T; {X})} =
  \begin{cases}
    \left( \int_0^T \| f(s) \|_{{X}}^r ds \right)^{1/r}, \quad 1 \leq r < \infty, \\
    \esssup_{t \in (0, T)} \| f (t) \|_{X}, \quad r = \infty.
  \end{cases}
}
We define $W^{k,r}(0, T; {X})$ for a non-negative integer $k$ and $1 \leq r \leq \infty$ as the
closure of $C^k ([0, T]; {X})$ with the norm $\| f \|_{W^{k,r}(0, T;{X})} = \sum_{i=0}^k \| \partial^i f/\partial t^i \|_{L^r(0, T; {X})}$. The semi-norm $\| f \|_{\dot{W}^{k,r}(0, T;{X})}$ is defined by $\| f \|_{\dot{W}^{k,r}(0, T;{X})} = \| \partial^k f/\partial t^k \|_{L^r(0, T; {X})}$.

Finally, for a normed space $X$ with its norm $\| \cdot \|_X$ and functions $f_1, f_2 \in X$, $\| f_1, f_2 \|_X$ will be used to denote $\| f_1 \|_X + \| f_2 \|_X$, and $\| f_1, f_2, f_3 \|_X$ is defined similarly. 

\subsection{A metamaterial model of acoustic wave propagation} \label{subsec:model}
A system of equations for the acoustic wave propagation with velocity and pressure unknowns is 
\algns{
  \rho \frac{\pd \vb}{\pd t} + \grad p &= \bs{f} , \\
  \kappa^{-1} \frac{\pd p}{\pd t} + \div \vb &= g
}
with the density $\rho$ and the bulk modulus $\kappa$. In conventional material models
the coefficients $\rho$ and $\kappa$ are fixed uniformly positive functions in $\Omega$. In this paper we are interested in metamaterial models such that $\rho$ and $\kappa^{-1}$ are frequency-dependent, more precisely, the temporal Fourier transform of the equations with frequency $\omega$ satisfy
\algns{
  -i \omega \hat{\rho} (\omega) \hat{\vb}(\omega)  + \grad \hat{p}(\omega)  = \hat{\bs{f}}(\omega)  , \\
  -i \omega \hat{\kappa}^{-1}(\omega)  \hat{p}(\omega)  + \div \hat{\vb}(\omega)  = \hat{g}(\omega) 
}
with 
\algns{
\hat{\rho}(\omega) &= \rho_a \LRp{ 1 - \frac{\Omega_{\rho}^2}{\omega^2 - \omega_{\rho}^2 }} \qquad 
\hat{\kappa}^{-1}(\omega) &= \kappa_a^{-1} \LRp{ 1 - \frac{\Omega_{\kappa}^2}{\omega^2 - \omega_{\kappa}^2 + i \gamma \omega }} , \quad \gamma \ge 0
}
where $\rho_a, \kappa_a >0$ are functions in $\Omega$ with uniform positive lower bounds, $\Omega_{\rho} \ge 0$, $\Omega_{\kappa} \ge 0$ are functions in $\Omega$, $\omega_{\rho} >0$, $\omega_{\kappa}>0$, $\gamma \ge 0$ are constants in $\Omega$, and $\hat{\vb}$, $\hat{p}$, $\hat{\bs{f}}$, $\hat{g}$ are the temporal Fourier transforms of $\vb$, $p$, $\bs{f}$, $g$, respectively.

To obtain a system of time-dependent equations we introduce new variables $\ub$, $\wb$, $q$, $r$ satisfying
\algns{
  i\omega \hat{\vb} (\omega) &= (\omega^2 - \omega_{\rho}^2) \hat{\ub} (\omega) , & -i \omega \hat{\wb} (\omega) &= \hat{\ub} (\omega) , \\
  i\omega \hat{p} (\omega) &= (\omega^2 + i \gamma \omega - \omega_{\kappa}^2) \hat{q} (\omega) , & -i\omega \hat{r} (\omega) &= \hat{q} (\omega) .
}
The system of time-dependent equations are 
\begin{subequations} \label{eq:original-eqs}
  \algn{
    \label{eq:original-eq-1}  \rho_a \frac{\pd \vb}{\pd t} + \grad p + \rho_a \Omega_{\rho}^2 \ub &= \fb , \\
	\label{eq:original-eq-2}  \kappa_a^{-1} \frac{\pd p}{\pd t} + \div \vb + \kappa_a^{-1} \Omega_{\kappa}^2 q &= g ,\\
	\frac{\pd \ub}{\pd t} - \vb + \omega_{\rho}^2 \wb &= 0, \\
	\frac{\pd \wb}{\pd t} - \ub &= 0, \\
	\frac{\pd q}{\pd t} - p + \gamma q + \omega_{\kappa}^2 r &= 0, \\
	\frac{\pd r}{\pd t} - q &= 0 .
  }
\end{subequations}
We assume that the material of wave propagation in $\Omega$ consists of a conventional positive index material (PIM) on a subdomain $\Omega_P \subset \Omega$ and a negative index material (NIM) on $\Omega \setminus \overline{\Omega_P}$. The PIM and NIM materials are mathematically modeled by the values of $\Omega_{\rho}$ and $\Omega_{\kappa}$, i.e., $\Omega_{\rho} = \Omega_{\kappa} = 0$ on $\Omega_P$ and $\Omega_{\rho}, \Omega_{\kappa} >0 $ on $\Omega \setminus \overline{\Omega_P}$. Note that the first two equations in \eqref{eq:original-eqs} are decoupled from the other equations on the domain $\Omega_P$ because $\Omega_{\rho} = \Omega_{\kappa} = 0$ on $\Omega_P$. From this observation we may develop numerical methods which solve different sets of equations on the PIM and NIM domains. However, we will focus on a monolithic numerical method for the system because monolithic approaches can cover problems with varying interfaces between PIM and NIM in a unified manner. They can be used for shape optimization problems for metamaterial device design, which is our future research interest.

For boundary conditions of \eqref{eq:original-eqs} let $\Gamma_D$, $\Gamma_N$ be the subsets of $\pd \Omega$ such that $\Gamma_D \cap \Gamma_N = \emptyset$, $\overline{\Gamma_D} \cup \overline{\Gamma_N} = \pd \Omega$. Then imposed boundary conditions are 
\algn{ \label{eq:general-bc}
	p (t) = p_D (t) \quad \text{ on } \Gamma_D, \qquad \vb(t) \cdot \bs{n} = v_N (t) \quad \text{ on } \Gamma_N
}
with given functions $p_D$ on $(0,T]\times \Gamma_D$ and $v_N$ on $(0,T]\times \Gamma_N$, where $\bs{n}$ is the unit outward normal vector field on $\Gamma_N$. 

To write a variational form of the system let us define function spaces
\algns{
  \Wb = L^2(\Omega; \R^d), \qquad Q = L^2(\Omega), \qquad \Vb = \{ \vb \in \Wb \,:\, \div \vb \in L^2(\Omega) \} 
}
where $\div \vb$ is defined in the sense of distributions. 
For future reference we define $\mathcal{X} := \Vb \times Q \times \Wb \times \Wb \times Q \times Q$ with the norm induced by the $L^2$ norms of the function spaces. We also define $\rho_u$, $\rho_w$, $\rho_q$, $\rho_r$ to denote (nonnegative) weights 
\algns{
  \rho_u = \rho_{a}\Omega_{\rho}^2, \quad \rho_w = \rho_{a}\omega_{\rho}^2 \Omega_{\rho}^2, \quad 
  \rho_q = \kappa_a^{-1} \Omega_{\kappa}^2, \quad \rho_r = \kappa_a^{-1} \omega_{\kappa}^2 \Omega_{\kappa}^2
}
in the rest of this paper.

For well-posedness of \eqref{eq:original-eqs} we recall the following result from \cite[Theorem~3.1]{Bellis-Lombard:2019}.
\begin{theorem} \label{thm:well-posed}
  For \eqref{eq:original-eqs} suppose that initial data $(\vb(0), p(0), \ub(0), \wb(0), q(0), r(0)) \in \mc{X}$ satisfy $\vb(0), \ub(0), \wb(0) \in H^1(\Omega; \R^d)$, $p(0), q(0), r(0) \in H^1(\Omega)$. In addition, suppose that $\fb \in C^1([0,T]; \Wb)$, $g \in C^1([0,T]; Q)$ hold. Then there exists a unique solution 
\algn{ \label{eq:solution-regularity}
  (\vb, p, \ub, \wb, q, r) \in C^1([0,T]; \mc{X}) \cap C^0([0,T]; \mc{X})
}
for the given initial data and $\fb$, $g$. 
\end{theorem}

For finite element discretization we need to consider a variational form of \eqref{eq:original-eqs}. For simplicity of presentation we assume the homogeneous boundary condition
\algn{ \label{eq:homog-bc}
  p(t) = 0 \quad \text{ on } \pd \Omega 
}
for all $t \in (0,T]$. 

For simplicity of presentation we will use $\dot{v}$ instead of $\partial v / \partial t$ in the rest of paper. 
\begin{definition}
 For $\fb \in L^1((0,T); \Wb)$, $g \in L^1((0,T); Q)$, we say $(\vb, p, \ub, \wb, q, r) \in H^1([0,T], \mathcal{X}) \cap L^2((0,T); \mathcal{X})$ a weak solution of \eqref{eq:original-eqs} if it satisfies 
\begin{subequations} \label{eq:wave-cont1}
  \algn{
    \label{eq:wave-cont1a} \LRp{\rho_{a} \dot \vb , \vb'} - \LRp{ p, \div \vb'} + \LRp{ \rho_{u} \ub, \vb'} &= \LRp{ \fb , \vb'}, \\
    \LRp{ \kappa_a^{-1} \dot p, p'} + \LRp{ \div \vb, p'} + \LRp{\rho_q q, p'} &= \LRp{ g, p'} ,\\
    \LRp{  \dot \ub , \ub'} - \LRp{ \vb, \ub'} + \LRp{ \omega_{\rho}^2 \wb, \ub'} &= 0, \\
    \LRp{ \dot \wb , \wb'} - \LRp{ \ub, \wb'} &= 0, \\
    \LRp{ \dot q, q'} - \LRp{ p, q'} + \LRp{ \gamma  q, q'} + \LRp{\omega_{\kappa}^2 r , q'} &= 0, \\
    \LRp{ \dot r, r'} - \LRp{  q, r'} &= 0 
  }
\end{subequations}
for $(\vb', p', \ub', \wb', q', r') \in \mathcal{X}$ and for almost every $t \in (0,T)$.
\end{definition}
One can easily check by the integration by parts that the solution in Theorem~\ref{thm:well-posed} with the boundary condition \eqref{eq:homog-bc} is a weak solution satisfying \eqref{eq:wave-cont1}. 

We remark that the above variational form can cover general boundary conditions with some necessary modifications. For the boundary condition \eqref{eq:general-bc} we replace $\Vb$ by 
\algns{
\Vb_N = \{ \vb \in \Wb \,:\, \div \vb \in L^2(\Omega), \vb \cdot \bs{n} = v_N \text{ on } \Gamma_N \}
}
and replace \eqref{eq:wave-cont1a} by 
\algns{
\LRp{\rho_{a} \dot \vb , \vb'} - \LRp{ p, \div \vb'} + \LRp{ \rho_{u} \ub, \vb'} &= \LRp{ \fb , \vb'}  - \int_{\Gamma_D} p_D \vb' \cdot \bs{n} \,ds 
}
for the test function $\vb'$ in $\Vb_N^0 := \{ \vb \in \Wb \,:\, \div \vb \in L^2(\Omega), \vb \cdot \bs{n} = 0 \text{ on } \Gamma_N \}$.

\begin{theorem}
  If $(\vb, p, \ub, \wb, q, r)$ is a solution of \eqref{eq:original-eqs} satisfying \eqref{eq:solution-regularity}, then 
  \mltln{ \label{eq:integral-ineq-1}
  \| \vb, p, \ub, \wb, q, r \|_{L^\infty((0,T); \mathcal{X})} \\
  \le C_1 \| \vb(0), p(0), \ub(0), \wb(0), q(0), r(0) \|_{\mathcal{X}} + C_2 \| \fb, g \|_{L^1((0,T); \Wb \times Q)} 
  }
  holds with $C_2$ which may depend on $T$. Moreover, if we define
  \algn{ \label{eq:E0-def}
    E_0(t)^2 = \| \ub(t) \|_{\rho_u}^2 +  \| \vb(t) \|_{\rho_a}^2 +  \| \wb(t) \|_{\rho_w}^2 +  \| p(t) \|_{\kappa_a^{-1}}^2 +  \| q(t) \|_{\rho_q}^2 +  \| r(t) \|_{\rho_r}^2
  }
  with the weighted (semi)-norm $\| \cdot \|_{\rho}$, $\rho = \rho_u, \rho_a, \rho_w, \kappa_a^{-1}, \rho_q, \rho_r$,
  then 
  \algn{ \label{eq:integral-ineq-2}
    E_0(t) \le E_0(0) + C \int_0^t \LRp{ \| \fb(s) \|_0 + \| g(s) \|_0 } \,ds .
  }
  with $C>0$ depending only on $\| \rho_a^{-1} \|_{L^\infty}$ and $\| \kappa_a \|_{L^\infty}$.
\end{theorem}
\begin{proof}
Recall that $(\vb, p, \ub, \wb, q, r)$ satisfies \eqref{eq:wave-cont1}. 
If we choose $(\vb', p', \ub', \wb', q', r') = (\vb, p, \ub, \wb, q, r)$ in \eqref{eq:wave-cont1} and add all the equations, then we get 
\algn{ \label{eq:energy-eq1}
&\frac 12 \frac{d}{dt} \LRp{ \| \ub \|_{0}^2 +  \| \vb \|_{\rho_a}^2 +  \| \wb \|_{0}^2 +  \| p \|_{\kappa_a^{-1}}^2 +  \| q \|_{0}^2 +  \| r \|_{0}^2 } + \LRp{ \gamma q, q} \\
\notag &+ \LRp{ (\rho_{u} -1 )\ub, \vb} + \LRp{ (\rho_q -1 ) q, p} \\
\notag &+ \LRp{ (\omega_{\rho}^2 -1) \wb, \ub } + \LRp{ (\omega_{\kappa}^2 - 1) r, q} = \LRp{ \fb, \vb} + \LRp{  g, p} .
}
Let $E_1(t)^2 = \| \ub(t) \|_{0}^2 +  \| \vb(t) \|_{\rho_a}^2 +  \| \wb(t) \|_{0}^2 +  \| p(t) \|_{\kappa_a^{-1}}^2 +  \| q(t) \|_{0}^2 +  \| r(t) \|_{0}^2$. 
The Cauchy--Schwarz inequality with the above identity gives 
\algns{
\frac{d}{dt} E_1(t)^2 \le C E_1(t)^2 + (\| \fb (t) \|_0 + \| g (t) \|_0) E_1(t)
}
with $C$ depending on $\rho_u$, $\rho_q$, $\omega_{\rho}$, $\omega_{\kappa}$. 
By Gronwall lemma one can obtain 
\algns{
  E_1(t) \le E_1(0) + C (t) \int_0^t (\| \fb(s) \|_0 + \| g(s) \|_0 ) \,ds .
}
Then \eqref{eq:integral-ineq-1} follows from the equivalence of 
$$\sqrt{E_1(t)} \text{ and } \|(\ub(t), \vb(t), \wb(t), p(t), q(t), r(t)) \|_{\mc{X}}.$$

To prove \eqref{eq:integral-ineq-2} we choose $(\vb', p', \ub', \wb', q', r') = (\vb, p, \rho_u \ub, \rho_w  \wb, \rho_q q, \rho_r r)$ in \eqref{eq:wave-cont1} and add all the equations. Then we get 
\algns{
  \frac 12 \frac{d}{dt} E_0(t)^2 + \LRp{ \gamma \rho_q q, q}  = \LRp{ \fb, \vb} + \LRp{ g, p} .
}

If $E_0(t) \le E_0(0)$, then there is nothing to prove, so we assume $E_0(t) > E_0(0)$ and will prove \eqref{eq:integral-ineq-2} in the rest of the proof.

First, we prove \eqref{eq:integral-ineq-2} assuming that $E_0(t) = \esssup_{s \in [0,t]} E_0(s)$. 
Since $\LRp{ \gamma \rho_q q, q} \ge 0$, integration of the above identity from $0$ to $t$ gives 
\algns{
  E_0(t)^2 - E_0(0)^2 &\le 2 \max \{\| \rho_a^{-1} \|_{L^\infty}, \| \kappa_a \|_{L^\infty} \} \int_0^t \LRp{ \| \fb(s) \|_0 + \| g (s) \|_0 } E_0(s) \,ds \\
  &\le 2 \max \{\| \rho_a^{-1} \|_{L^\infty}, \| \kappa_a \|_{L^\infty} \} \int_0^t \LRp{ \| \fb(s) \|_0 + \| g (s) \|_0 } \,ds E_0(t) .
}
Then 
\algn{ 
  \notag E_0(t) &\le \frac{E_0(0)^2}{E_0(t)} + 2 \max \{\| \rho_a^{-1} \|_{L^\infty}, \| \kappa_a \|_{L^\infty} \} \int_0^t \LRp{ \| \fb(s) \|_0 + \| g \|_0 } \,ds  \\
  \label{eq:E0-estm-1} &\le  E_0(0) + 2 \max \{\| \rho_a^{-1} \|_{L^\infty}, \| \kappa_a \|_{L^\infty} \} \int_0^t \LRp{ \| \fb(s) \|_0 + \| g \|_0 } \,ds 
}
which proves \eqref{eq:integral-ineq-2}.

If $0<E_0(t) < \esssup_{s \in [0,t]} E_0(s)$, then there exists $0 \le t_0 <t$ such that $E_0(t_0) = \esssup_{s \in [0,t_0]} E_0(s)$ and $E_0(t) < E_0(t_0)$. By the same argument as above, we can obtain the inequality \eqref{eq:E0-estm-1} for $E_0(t_0)$. Then 
\algns{
  E_0(t) &< E_0(t_0) \le E_0(0) + 2 \max \{\| \rho_a^{-1} \|_{L^\infty}, \| \kappa_a \|_{L^\infty} \} \int_0^{t_0} \LRp{ \| \fb(s) \|_0 + \| g (s) \|_0 } \,ds \\
  &\le E_0(0) + 2 \max \{\| \rho_a^{-1} \|_{L^\infty}, \| \kappa_a \|_{L^\infty} \} \int_0^{t} \LRp{ \| \fb(s) \|_0 + \| g (s) \|_0 } \,ds ,
}
so \eqref{eq:integral-ineq-2} is proved.
\end{proof}
By observing \eqref{eq:original-eq-1} and \eqref{eq:original-eq-2}, the auxiliary variables $(\ub, \wb, q, r)$ interact with $(\vb, p)$ only on the NIM domain on which $\Omega_{\rho}$ and $\Omega_{\kappa}$ are strictly positive. In fact, the physical meaning of $(\ub, \wb, q, r)$ on $\Omega_P$ is not clear, so there is no natural way to determine the initial data of $(\ub, \wb, q, r)$ on $\Omega_P$.
In the following theorem we show that $(\vb, p)$ in \eqref{eq:wave-cont1} is independent on the initial data of $(\ub, \wb, q, r)$ on $\Omega_P$.
As a consequence, any choice of initial data $(\ub, \wb, q, r)$ on $\Omega_P$ is allowed to obtain a unique $(\vb, p)$. 
This argument can be extended to our numerical scheme, so there is no concern in the choice of numerical initial data of $(\ub, \wb, q, r)$ on $\Omega_P$.
%
\begin{theorem}
  Given $\fb \in L^1((0,T); \Wb)$, $g \in L^1((0,T); Q)$, and initial data ${\bf U}(0) \in \mathcal{X}$, \eqref{eq:original-eqs} has a unique weak solution. In addition, suppose that ${\bf U}_i \in H^1([0,T]; \mathcal{X}) \cap L^2((0,T); \mathcal{X})$, $i=1,2$ are the weak solutions for the two sets of initial data ${\bf U}_i(0) \in \mathcal{X}$, $i=1,2$. For ${\bf U}_i(t):=(\vb_i(t), p_i(t), \ub_i(t), \wb_i(t), q_i(t), r_i(t))$ with $i=1,2$, if 
\algn{
\label{eq:initial-id-1} &\vb_1(0) = \vb_2(0), \quad p_1(0) = p_2(0) & & \text{ on }\Omega, \\
\label{eq:initial-id-2} &\ub_1(0) = \ub_2(0), \wb_1(0) = \wb_2(0), q_1(0) = q_2(0), r_1(0) = r_2(0) & & \text{ on }\Omega \setminus \overline{\Omega_P},  
}
then the same identities hold for the weak solutions ${\bf U}_1(t)$, ${\bf U}_2(t)$ for $t \in(0,T]$.
\end{theorem}
\begin{proof}
  If ${\bf U}_i$, $i=1,2$ are weak solutions for given $\fb$, $g$, and initial data ${\bf U}(0) \in \mathcal{X}$. Then ${\bf U}_1 - {\bf U}_2$ is a weak solution for $\fb = \bs{0}$, $g=0$, and zero initial data. Then ${\bf U}_1 = {\bf U}_2$ follows by \eqref{eq:integral-ineq-1}. 
  
To prove the second part of the assertion, let ${\bf U}_i$, $i=1,2$ be the weak solutions for the initial data 
\algns{
(\vb_i(0), p_i(0), \ub_i(0), \wb_i(0), q_i(0), r_i(0)), \quad i=1,2
}
satisfying \eqref{eq:initial-id-1} and \eqref{eq:initial-id-2}. Let $(\vb, p, \ub, \wb, q, r)$ be the difference ${\bf U}_1 - {\bf U}_2$ and define $E_0(t)$ as in \eqref{eq:E0-def}. Then $E_0(t)$ satisfies \eqref{eq:integral-ineq-2} with $\fb = \bs{0}$ and $g=0$. Moreover, $E_0(0) = 0$ because of \eqref{eq:initial-id-1}, \eqref{eq:initial-id-2}, so $E_0(t) = 0$ holds for all $t \in (0,T]$ and the assertion follows. 
\end{proof}

\section{Finite elements discretization and error analysis } \label{sec:error-analysis}
   
\subsection{Finite elements for spatial discretization } \label{subsec:discretization}

Recall that $\mc{T}_h$ is a triangulation of $\Omega$ without hanging nodes. In the rest of this paper we assume that the density functions $\rho_{\sigma}$ with $\sigma = u, w, q, r$ are in $W_h^{1,\infty}(\mc{T}_h)$ where 
\algns{
W_h^{1,\infty}(\mc{T}_h) := \{ \rho \in L^2(\Omega) \,:\, \rho|_K \in L^{\infty}(K), \; \grad (\rho|_K) \in L^{\infty}(K; \R^d)  \quad \forall K \in \mc{T}_h \} 
}
with the norm $\| \rho \|_{W_h^{1,\infty}} := \sup_{K \in \mc{T}_h} (\| \rho|_K \|_{L^{\infty}(K)} + \| \grad (\rho|_K) \|_{L^{\infty}(K)})$.

Finite element discretization of the first order differential equation form of acoustic wave equations, is studied in \cite{Geveci:1988}. We extend the approach in \cite{Geveci:1988} to include the auxiliary variables.
For discretization of \eqref{eq:wave-cont1} with finite elements we use finite element spaces
$\Vb_h \subset \Vb$, $\Wb_h \subset \Wb$, $Q_h \subset Q$ which are defined below. First,
${\rm BDM}_l(K)$ for $l \ge 1$ and ${\rm RTN}_l (K)$ for $l \ge 0$ are defined by
\algns{
 {\rm BDM}_l (K) = \{ v \in \mc{P}_{l}(K; \R^d) \},  \quad {\rm RTN}_{l} (K) = \{ v \in \mc{P}_{l}(K; \R^d) + \bs{x} \mc{P}_{l}(K) \} 
}
where $\bs{x} = (x_1, \ldots, x_d)^T$.

In the rest of this paper $k\ge 0$ is a fixed integer. For given $k\ge 0$ we set $\bs{S}_k(K)$ as either ${\rm BDM}_{k+1} (K)$ or ${\rm RTN}_{k} (K)$, and define $\Vb_h$ as the finite element space
\algn{ \label{eq:Vh}
  \Vb_{h} &= \{ \vb \in \Vb \,:\, \vb|_K \in \bs{S}_k (K), \quad K \in \mc{T}_h \}. 
}
By this definition $\Vb_h$ is the Brezzi--Douglas--Marini or the N\'{e}d\'{e}lec element of the second kind if $\bs{S}_k(K)= {\rm BDM}_{k+1}(K)$ \cite{BDM85,Nedelec86}, and is the Raviart--Thomas or the N\'{e}d\'{e}lec element of the first kind if $\bs{S}_k(K)={\rm RTN}_k(K)$ \cite{Nedelec80,RT75}. For the details on the definition of $\Vb_h$ with $\bs{S}_k(K)= {\rm BDM}_{k+1} (K)$ or $\bs{S}_k(K) = {\rm RTN}_{k} (K)$, we refer to \cite{Brezzi-Fortin-book,Boffi-Brezzi-Fortin-book} and the original articles \cite{BDM85,Nedelec80,Nedelec86,RT75}. For $\Wb_h$, let $\Wb_h(K)$ be 
\algns{
	\Wb_h(K) = 
	\begin{cases}
		\mc{P}_{k+1}(K;\R^d)  & \text{ if } \bs{S}_k(K) = {\rm BDM}_{k+1} (K), \\
		\mc{P}_k(K;\R^d) & \text{ if } \bs{S}_k(K) = {\rm RTN}_{k} (K) ,
	\end{cases}
}
and define $\Wb_h$, $Q_h$ as 
\algn{
  \label{eq:Wh} \Wb_h &= \{ \vb \in L^2(\Omega; \R^d) \,:\, \vb|_K \in \Wb_h(K) \}, \\
  \label{eq:Qh} Q_h &= \{ q \in L^2(\Omega)\,:\, q|_K \in \mc{P}_k(K) \} .
}

Let $\mc{X}_h$ be $\Vb_h \times Q_h \times \Wb_h \times \Wb_h \times Q_h \times Q_h$. 
For $\fb \in C^0([0,T]; \Wb)$, $g \in C^0([0,T]; Q)$ the semidiscrete problem is to 
find $(\vb_h, p_h, \ub_h, \wb_h, q_h, r_h) \in C^1((0,T], \mc{X}_h)$ which satisfies 
\begin{subequations} \label{eq:wave-semidiscrete}
  \algn{
    \LRp{\rho_{a} \dot \vb_h , \vb'} - \LRp{ p_h, \div \vb'} + \LRp{ \rho_{u} \ub_h, \vb'} &= \LRp{ \fb , \vb'}, \\
    \LRp{ \kappa_a^{-1} \dot p_h, p'} + \LRp{ \div \vb_h, p'} + \LRp{\rho_q q_h, p'} &= \LRp{ g, p'} ,\\
    \LRp{  \dot \ub_h , \ub'} - \LRp{ \vb_h, \ub'} + \LRp{ \omega_{\rho}^2 \wb_h, \ub'} &= 0, \\
    \LRp{ \dot \wb_h , \wb'} - \LRp{ \ub_h, \wb'} &= 0, \\
    \LRp{ \dot q_h, q'} - \LRp{ p_h, q'} + \LRp{ \gamma  q_h, q'} + \LRp{\omega_{\kappa}^2 r_h , q'} &= 0, \\
    \LRp{ \dot r_h, r'} - \LRp{  q_h, r'} &= 0 
  }
\end{subequations}
for $(\vb', p', \ub', \wb', q', r') \in \mathcal{X}_h$ and for all $t \in (0,T]$. 

We will not discuss an error analysis for semidiscrete solutions in the paper. 
Instead, we will show a detailed error analysis of fully discrete solutions in the subsection below.

\subsection{Error analysis of fully discrete solutions} \label{subsec:error-analysis}

In this subsection we consider fully discrete solutions of \eqref{eq:wave-semidiscrete} with the Crank--Nicolson scheme and show the a priori error estimates.

For $T>0$ let $\Delta t = T/N$ for a natural number $N$ and define $\{t_n\}_{n=0}^N$ by $t_n = n \Delta t$. 
For a variable $\sigma:[0,T] \ra X$ for a Hilbert space $X$, we will use $\sigma_h^n$ and $\sigma^n$ for the numerical solution of $\sigma$ at $t_n$ and $\sigma(t_n)$, respectively. The variable $\sigma$ can be $\ub, \vb, \wb, p, q, r$ in the problem. As such, $\fb^n$ and $g^n$ will denote $\fb(t_n)$ and $g(t_n)$ for a time-dependent functions $\fb \in C^0([0,T]; \Wb)$ and $g\in C^0([0,T]; Q)$. For simplicity we will also use the definitions
\algns{
  \bar{\pd}_t v^{n+\frac 12} := \frac 1{\Delta t} \LRp{ v^{n+1} - v^n } , \qquad v^{n+\frac 12} := \frac 12 \LRp{ v^n + v^{n+1} } 
}
for any sequence $\{v^n\}_{n=0}^N$ with the upper index $n$. 

The Crank--Nicolson scheme of \eqref{eq:wave-cont1} is the following: For given 
\algns{
{\bf U}^n :=(\vb_h^n, p_h^n,\ub_h^n, \wb_h^n, q_h^n, r_h^n), \quad  \fb^n, \quad \fb^{n+1}, \quad g^n, \quad g^{n+1}, 
}
we find ${\bf U}_h^{n+1} := (\vb_h^{n+1}, p_h^{n+1}, \ub_h^{n+1}, \wb_h^{n+1}, q_h^{n+1}, r_h^{n+1}) \in \mc{X}_h$ such that 
\begin{subequations} \label{eq:wave-full-disc}
  \algn{
    \LRp{\rho_{a} \bar{\pd}_t \vb_h^{n+\frac 12} , \vb'} - \LRp{ p_h^{n+\frac 12}, \div \vb'} + \LRp{ \rho_{u} \ub_h^{n+\frac 12}, \vb'} &= \LRp{ \fb^{n+\frac 12} , \vb'}, \\
    \LRp{ \kappa_a^{-1} \bar{\pd}_t p_h^{n+\frac 12}, p'} + \LRp{ \div \vb_h^{n+\frac 12}, p'} + \LRp{\rho_q q_h^{n+\frac 12}, p'} &= \LRp{ g^{n+\frac 12}, p'} ,\\
    \LRp{ \bar{\pd}_t \ub_h^{n+\frac 12} , \ub'} - \LRp{ \vb_h^{n+\frac 12}, \ub'} + \LRp{\omega_{\rho}^2 \wb_h^{n+\frac 12}, \ub'} &= 0, \\
    \LRp{ \bar{\pd}_t \wb_h^{n+\frac 12} , \wb'} - \LRp{ \ub_h^{n+\frac 12}, \wb'} &= 0, \\
    \LRp{ \bar{\pd}_t q_h^{n+\frac 12}, q'} - \LRp{ p_h^{n+\frac 12}, q'} + \LRp{ \gamma q_h^{n+\frac 12}, q'} + \LRp{\omega_{\kappa}^2 r_h^{n+\frac 12}, q'} &= 0, \\
    \LRp{ \bar{\pd}_t r_h^{n+\frac 12}, r'} - \LRp{ q_h^{n+\frac 12}, r'} &= 0 
  }
\end{subequations}
for all $(\vb', p', \ub', \wb', q', r') \in \mc{X}_h$.

For the well-definedness of this fully discrete scheme, we show that ${\bf U}_h^{n+1} = \bs{0}$ if 
\algn{ \label{eq:vanish-assumption}
{\bf U}_h^n = \bs{0}, \qquad \fb^n = \fb^{n+1} = \bs{0}, \qquad g^n = g^{n+1} = 0 .
}
To show it, assume that \eqref{eq:vanish-assumption} is true. Then \eqref{eq:wave-full-disc} becomes 
\begin{subequations} \label{eq:wave-well-posed-eqs}
\algn{ 
    \label{eq:wave-well-posed-eq-1} \frac 1{\Delta t} \LRp{\rho_{a} \vb_h^{n+1} , \vb'} - \frac 12 \LRp{ p_h^{n+1}, \div \vb'} + \frac 12 \LRp{ \rho_{u} \ub_h^{n+1}, \vb'} &= 0, \\
    \label{eq:wave-well-posed-eq-2} \frac 1{\Delta t} \LRp{ \kappa_a^{-1} p_h^{n+1}, p'} + \frac 12 \LRp{ \div \vb_h^{n+1}, p'} + \frac 12 \LRp{\rho_q q_h^{n+1}, p'} &= 0 ,\\
    \label{eq:wave-well-posed-eq-3} \frac 1{\Delta t} \LRp{ \ub_h^{n+1} , \ub'} - \frac 12 \LRp{ \vb_h^{n+1}, \ub'} + \frac 12 \LRp{\omega_{\rho}^2 \wb_h^{n+1}, \ub'} &= 0, \\
    \label{eq:wave-well-posed-eq-4} \frac 1{\Delta t} \LRp{ \wb_h^{n+1} , \wb'} - \frac 12 \LRp{ \ub_h^{n+1}, \wb'} &= 0, \\
    \label{eq:wave-well-posed-eq-5} \frac 1{\Delta t} \LRp{ q_h^{n+1}, q'} - \frac 12 \LRp{ p_h^{n+1}, q'} + \frac 12 \LRp{ \gamma q_h^{n+1}, q'} + \frac 12 \LRp{\omega_{\kappa}^2 r_h^{n+1}, q'} &= 0, \\
    \label{eq:wave-well-posed-eq-6} \frac 1{\Delta t} \LRp{ r_h^{n+1}, r'} - \frac 12 \LRp{ q_h^{n+1}, r'} &= 0 .
}
\end{subequations}
Let $P_h$ and $\bs{P}_h$ be the standard $L^2$ projections into $Q_h$ and $\Wb_h$. 
If we take 
\gats{
  \vb' = \vb_h^{n+1}, \qquad p' = p_h^{n+1}, \qquad \ub' = \bs{P}_h (\rho_u \ub_h^{n+1}), \\
  \wb' = \omega_{\rho}^2 \bs{P}_h (\rho_u \wb_h^{n+1}), \qquad q' = P_h(\rho_q q_h^{n+1}), \qquad 
  r' = \omega_{\kappa}^2 P_h (\rho_q r_h^{n+1}) ,
}
in \eqref{eq:wave-well-posed-eq-1} and add all the equations, then we get 
\algns{
& \frac 1{\Delta t} \LRp{ \| \vb_h^{n+1} \|_{\rho_a}^2 + \| p_h^{n+1} \|_{\kappa_a^{-1}}^2 + \| \ub_h^{n+1} \|_{\rho_u}^2 + \| \wb_h^{n+1} \|_{\rho_w}^2 + \| q_h^{n+1} \|_{\rho_q}^2 + \| r_h^{n+1} \|_{\rho_r}^2  } \\
&\quad  + \LRp{ \gamma \omega_{\kappa}^2 \rho_q r_h^{n+1}, r_h^{n+1} }  = 0, 
}
so $\vb_h^{n+1} = \bs{0}$, $p_h^{n+1} = 0$. From these, $\ub_h^{n+1} = \wb_h^{n+1} = \bs{0}$ follows by taking $\ub' = \ub_h^{n+1}$ and $\wb' = \omega_{\rho}^2 \wb_h^{n+1}$ in \eqref{eq:wave-well-posed-eq-3} and \eqref{eq:wave-well-posed-eq-4}, and then by adding them. Finally, $q_h^{n+1} = r_h^{n+1} = 0$ follows by taking $q' = q_h^{n+1}$ and $r' = \omega_{\kappa}^2 r_h^{n+1}$ in \eqref{eq:wave-well-posed-eq-5} and \eqref{eq:wave-well-posed-eq-6}, and then by adding them. Therefore, ${\bf U}_h^{n+1} = {\bf 0}$.

For the error analysis we use $e_{\sigma}^n = \sigma^n - \sigma_h^n$ for the error of variable $\sigma$ ($\sigma = \vb, \ub, \wb, p, q, r$) at $t = t_n$. For error equations we consider the difference of the average of \eqref{eq:wave-cont1} at $t_n$ and $t_{n+1}$, and the fully discrete scheme \eqref{eq:wave-full-disc}. Then the error equations are 
\begin{subequations} 
  \algns{
    \LRp{\rho_{a} (\dot{\vb}^{n+\frac 12} - \bar{\pd}_t \vb_h^{n+\frac 12} ), \vb'} - \LRp{ e_{p}^{n+\frac 12}, \div \vb'} + \LRp{ \rho_{u} e_{\ub}^{n+\frac 12}, \vb'} &= 0, \\
    \LRp{\kappa_a^{-1} (\dot{p}^{n+\frac 12} - \bar{\pd}_t p_h^{n+\frac 12} ), p'} + \LRp{ \div e_{\vb}^{n+\frac 12}, p'} + \LRp{\rho_q e_{q}^{n+\frac 12}, p'} &= 0 ,\\
    \LRp{\dot{\ub}^{n+\frac 12} - \bar{\pd}_t \ub_h^{n+\frac 12} , \ub'} - \LRp{ e_{\vb}^{n+\frac 12}, \ub'} + \LRp{\omega_{\rho}^2 e_{\wb}^{n+\frac 12}, \ub'} &= 0, \\
    \LRp{\dot{\wb}^{n+\frac 12} - \bar{\pd}_t \wb_h^{n+\frac 12} , \wb'} - \LRp{ e_{\ub}^{n+\frac 12}, \wb'} &= 0, \\
    \LRp{\dot{q}^{n+\frac 12} - \bar{\pd}_t q_h^{n+\frac 12} , q'} - \LRp{ e_{p}^{n+\frac 12}, q'} + \LRp{ \gamma  e_{q}^{n+\frac 12}, q'} + \LRp{\omega_{\kappa}^2 e_{r}^{n+\frac 12}, q'} &= 0, \\
    \LRp{\dot{r}^{n+\frac 12} - \bar{\pd}_t r_h^{n+\frac 12} , r'} - \LRp{ e_{q}^{n+\frac 12}, r'} &= 0 .
  }
\end{subequations}

For $\vb' \in H^s(\Omega; \R^d), s > \frac 12$, we define $\Pi_h$ as the canonical interpolation operators of RTN or BDM element which satisfy
\algn{ \label{eq:commute-interpolation}
\div \Pi_h \vb' = P_h \div \vb', \qquad 
\| \vb' - \Pi_h \vb' \|_0 \le Ch^{m} \| \vb' \|_{m} 
}
with $m:= \max \{s, k+1+\delta \}$ where $\delta = 1$ if $\Vb_h$ is a BDM element and $\delta = 0$ if $\Vb_h$ is an RTN element. 

Using $\Pi_h$, $\bs{P}_h$, $P_h$, we can define the decomposition of errors 
\algn{
  \label{eq:err-decomp-v} e_{\vb}^n &= e_{\vb}^{I,n} + e_{\vb}^{h,n} := (\vb^n - \Pi_h \vb^n) + (\Pi_h \vb^n - \vb_h^n), \\
  \label{eq:err-decomp-u} e_{\ub}^n &= e_{\ub}^{I,n} + e_{\ub}^{h,n} := (\ub^n - \bs{P}_h \ub^n) + (\bs{P}_h \ub^n - \ub_h^n), \\
  \label{eq:err-decomp-w} e_{\wb}^n &= e_{\wb}^{I,n} + e_{\wb}^{h,n} := (\wb^n - \bs{P}_h \wb^n) + (\bs{P}_h \wb^n - \wb_h^n), \\
  \label{eq:err-decomp-p} e_{p}^n &= e_{p}^{I,n} + e_{p}^{h,n} := (p^n - P_h p^n) + (P_h p^n - p_h^n), \\
  \label{eq:err-decomp-q} e_{q}^n &= e_{q}^{I,n} + e_{q}^{h,n} := (q^n - P_h q^n) + (P_h q^n - q_h^n), \\
  \label{eq:err-decomp-r} e_{r}^n &= e_{r}^{I,n} + e_{r}^{h,n} := (r^n - P_h r^n) + (P_h r^n - r_h^n).
}
For estimates of the interpolation errors denoted by $e_{\sigma}^{I,n}$ for a variable $\sigma$, let us use a generic symbol $I_h \sigma^n$ to denote the interpolation of the exact solution $\sigma^n$ into the corresponding finite element space. More specifically, $I_h = \Pi_h$ if $\sigma = \vb$, $I_h = \bs{P}_h$ if $\sigma = \ub, \wb$, and $I_h = P_h$ if $\sigma = p, q, r$. Then it holds that 
\algn{ \label{eq:approx-assumption}
 \| e_{\sigma}^{I,n} \|_0 = \| \sigma^n - I_h \sigma^n \|_0 \le Ch^s \| \sigma^n \|_{s} , 
}
with 
\algn{ 
\label{eq:s-range-1} &\frac 12 < s \le k + 1 + \delta & & \text{ if }\sigma = \vb, \\
\label{eq:s-range-2} &0 \le s \le k + 1 + \delta & & \text{ if }\sigma = \ub, \wb, \\
\label{eq:s-range-3} &0 \le s \le k+1 & &\text{ if }\sigma = p, q, r. 
}

By \eqref{eq:commute-interpolation} we can obtain 
\algns{
\LRp{e_p^{h,n}, \div \vb'} = 0 \quad \forall \vb' \in \Vb_h, \qquad \LRp{ \div e_{\vb}^{h,n}, q'} = 0 \quad \forall q' \in Q_h. 
}
By these identities, the orthogonality of $L^2$ projections, and some algebraic manipulations, the previous error equations are reduced to 
\begin{subequations} \label{eq:full-error-eqs}
  \algn{
    \label{eq:full-error-eq-1}\LRp{\rho_{a} \bar{\pd}_t e_{\vb}^{h,n+\frac 12} , \vb'} - \LRp{ e_{p}^{h,n+\frac 12}, \div \vb'} + \LRp{ \rho_{u} e_{\ub}^{h,n+\frac 12}, \vb'} &= F_v^n (\vb') , \\
    \LRp{ \kappa_a^{-1} \bar{\pd}_t e_{p}^{h,n+\frac 12}, p'} + \LRp{ \div e_{\vb}^{h,n+\frac 12}, p'} + \LRp{\rho_q e_{q}^{h,n+\frac 12}, p'} &= F_p^n (p') ,\\
    \LRp{ \bar{\pd}_t e_{\ub}^{h,n+\frac 12} , \ub'} - \LRp{ e_{\vb}^{h,n+\frac 12}, \ub'} + \LRp{\omega_{\rho}^2 e_{\wb}^{h,n+\frac 12}, \ub'} &= F_u^n(\ub') , \\
    \LRp{ \bar{\pd}_t e_{\wb}^{h,n+\frac 12} , \wb'} - \LRp{ e_{\ub}^{h,n+\frac 12}, \wb'} &= F_w^n (\wb'), \\
    \LRp{ \bar{\pd}_t e_{q}^{h,n+\frac 12}, q'} - \LRp{ e_{p}^{h,n+\frac 12}, q'} + \LRp{ \gamma  e_{q}^{h,n+\frac 12}, q'} + \LRp{\omega_{\kappa}^2 e_{r}^{h,n+\frac 12}, q'} &= F_q^n(q') , \\
    \LRp{ \bar{\pd}_t e_{r}^{h,n+\frac 12}, r'} - \LRp{ e_{q}^{h,n+\frac 12}, r'} &= F_r^n(r') 
  }
\end{subequations}
where 
\begin{subequations} \label{eq:F-terms}
\algn{
	\label{eq:Fv} F_v^n (\vb') &= - \LRp{ \rho_a \LRp{ \Pi_h \bar{\pd}_t \vb^{n+\frac 12} - \dot \vb^{n+\frac 12} } , \vb'} - \LRp{ \rho_u e_{\ub}^{I,n+\frac 12} , \vb'}, \\
	\label{eq:Fp} F_p^n (p') &=  - \LRp{ \kappa_a^{-1} \LRp{ \bar{\pd}_t P_h p^{n+\frac 12} - \dot p^{n+\frac 12} } , p'} - \LRp{ \rho_q e_q^{I,n+\frac 12}  , p'} , \\
	\label{eq:Fu} F_u^n (\ub') &=  - \LRp{  \LRp{ \bar{\pd}_t \bs{P}_h \ub^{n+\frac 12} - \dot \ub^{n+\frac 12} } , \ub'} , \\
	\label{eq:Fw} F_w^n (\wb') &= - \LRp{  \LRp{ \bar{\pd}_t \bs{P}_h \wb^{n+\frac 12} - \dot \wb^{n+\frac 12} } , \wb'} , \\
	\label{eq:Fq} F_q^n (q') &= - \LRp{ \LRp{ \bar{\pd}_t P_h q^{n+\frac 12} - \dot q^{n+\frac 12} }, q'}  , \\
	\label{eq:Fr} F_r^n (r') &= - \LRp{ \LRp{ \bar{\pd}_t P_h r^{n+\frac 12} - \dot r^{n+\frac 12} }, r'} .
}
\end{subequations}
%
In the discussions below we will use $\mc{E}^n$ defined as 
\algn{ \label{eq:En}
(\mc{E}^{n})^2 &= \| e_\ub^{h,n} \|_{\rho_{u}}^2 +  \| e_\vb^{h,n} \|_{\rho_{a}}^2 +  \| e_\wb^{h,n} \|_{\rho_{w}}^2 +  \| e_p^{h,n} \|_{\kappa_a^{-1}}^2 +  \| e_q^{h,n} \|_{\rho_{q}}^2 +  \| e_r^{h,n} \|_{\rho_{r}}^2 .
}
\begin{prop} \label{prop:error-estm}
  For given $\fb \in C^0([0,T]; \Wb)$, $g \in C^0([0,T]; Q)$ and initial data $(\vb(0), \ub(0), \wb(0), p(0), q(0), r(0)) \in \mc{X}$ suppose that $(\vb, \ub, \wb, p, q, r)$ is a weak solution of \eqref{eq:wave-cont1}. Assume that numerical initial data
\algns{
  (\vb_h(0), p_h(0), \ub_h(0), \wb_h(0), q_h(0), r_h(0)) \in \mc{X}_h
}
satisfy 
\algn{ 
  \notag &\| \ub(0) - \ub_h(0) \|_{\rho_{u}} +  \|  \vb(0) - \vb_h(0) \|_{\rho_{a}} +  \|  \wb(0) - \wb_h(0) \|_{\rho_{w}} \\
	\label{eq:initial-assumption}  &\quad +  \| P_h p(0) - p_h(0) \|_{\rho_{p}} +  \| P_h q(0) - q_h(0) \|_{\rho_{q}} +  \| P_h r(0) - r_h(0) \|_{\rho_{r}} \\
  \notag &\qquad \le C_0' h^s , \quad \frac 12 < s \le k+1+\delta 
}
with $C_0'$ independent of $h$. We also assume that the exact solution $(\vb, \ub, \wb, p, q, r)$ and $\rho_a$, $\kappa_a^{-1}$, $\rho_u$, $\rho_w$, $\rho_q$, $\rho_r$ satisfy the regularity assumptions \eqref{eq:C0-depend}, \eqref{eq:C1-depend}, \eqref{eq:C2-depend} below.

If $\{(\vb_h^n, \ub_h^n, \wb_h^n, p_h^n, q_h^n, r_h^n)\}_{n=1}^N$ is a numerical solution obtained by the fully discrete scheme \eqref{eq:wave-full-disc}, then 
\algn{ \label{eq:main-error-estm}
  \mc{E}^n \le C_0 h^s +  C_1(\Delta t)^2 + C_2 h^s , \quad \frac 12 < s \le k+1 +\delta 
}
with constants $C_0$, $C_1$ $C_2$ such that $C_0$ depends on $C_0'$ in \eqref{eq:initial-assumption} and 
\algn{ \label{eq:C0-depend}
\| \rho_u, \rho_a, \rho_w \|_{L^\infty}< \infty, 
}
$C_1$ depends on 
%
\algn{ \label{eq:C1-depend}
\| \vb, \ub, \wb, p, q, r \|_{\dot{W}^{3,1}(0, T; L^2)}, \quad 
\| \rho_a, \rho_u, \rho_w, \rho_q, \rho_r \|_{L^\infty}, 
}
and $C_2$ depends on 
\algn{ 
\label{eq:C2-depend}
\| \dot{\vb} \|_{L^1(0, T; H^s)}, \quad \| \dot{p} \|_{L^1(0, T; H^{s_0})}, \quad 
\| \rho_u, \kappa_a^{-1} \|_{W_h^{1,\infty}}, \quad \| \kappa_a, \rho_a, \rho_w, \rho_q, \rho_r \|_{L^\infty}
}
with $s_0 = \max\{ 0, s-1\}$.
\end{prop}
\begin{proof}
Let us take $(\vb', p', \ub', \wb', q', r') \in \mc{X}_h$ as $$(e_\vb^{h,n+\frac 12}, e_p^{h,n+\frac 12}, \bs{P}_h (\rho_u e_\ub^{h,n+\frac 12}), \bs{P}_h (\rho_w e_\wb^{h,n+\frac 12}), P_h (\rho_q e_q^{h,n+\frac 12}), P_h(\rho_r e_r^{h,n+\frac 12}) )$$ in \eqref{eq:full-error-eqs} and add all the equations, then we can get 
\algn{ 
\label{eq:error-energy-eq} &\frac 12 \LRp{ (\mc{E}^{n+1})^2 - (\mc{E}^{n})^2 } + \Delta t \LRp{ \gamma \rho_q e_q^{h,n+\frac 12} , e_q^{h, n+\frac 12}} \\
\notag &\quad = \Delta t  \LRp{ F_v^n (e_{\vb}^{h, n+\frac 12} ) + F_p^n (e_{p}^{h, n+\frac 12} ) + F_u^n (\bs{P}_h (\rho_u e_\ub^{h,n+\frac 12}) )  } \\
\notag &\qquad + \Delta t \LRp{ F_w^n (\bs{P}_h (\rho_w e_\wb^{h,n+\frac 12}) ) + F_q^n (P_h (\rho_q e_q^{h,n+\frac 12}) ) + F_r^n (P_h(\rho_r e_r^{h,n+\frac 12}) )) } \\
\notag &\quad =: R^n .
}

The proof of \eqref{eq:main-error-estm} has three steps. 
In the first step, we shall prove 
\algn{ 
  \label{eq:Rn-estm} \| R^n \|_0 \le (C_{1,n} (\Delta t)^2 + C_{2,n} h^s ) (\mc{E}^n + \mc{E}^{n+1}) , \qquad \frac 12 < s \le k+1+\delta
}
with $C_{1,n}$ depending on 
\algn{ \label{eq:C1n-depend}
\| \vb, \ub, \wb, p, q, r \|_{\dot{W}^{3,1}(t_n, t_{n+1}; L^2)}, \quad \| \rho_a, \rho_u, \rho_w, \rho_q, \rho_r \|_{L^\infty}, 
}
and with $C_{2,n}$ depending on 
\algn{ 
\begin{split} \label{eq:C2n-depend}
\| \dot{\vb} \|_{L^1(t_n, t_{n+1}; H^s)}, \quad \| \dot{p} \|_{L^1(t_n, t_{n+1}; H^{s_0})}, \\
\| \rho_u, \kappa_a^{-1} \|_{W_h^{1,\infty}}, \quad \| \kappa_a, \rho_a, \rho_w, \rho_q, \rho_r \|_{L^\infty}
\end{split}
}
with $s_0 = \max\{ 0, s-1\}$. The more detailed dependence of $C_{1,n}$, $C_{2,n}$ will be clarified in the proof of \eqref{eq:Rn-estm}. In the second step, we prove 
\algn{ \label{eq:En-estm}
	\mc{E}^{n} \le \mc{E}^0 + 2 \sum_{i=0}^{n-1} (C_{1,i} (\Delta t)^2 + C_{2,i} h^s ) , \qquad \frac 12 < s \le k+1+\delta .
}
In the third step, we prove 
\algn{ \label{eq:E0-estm}
\mc{E}^0 \le C_0 h^s , \qquad \frac 12 < s \le k+1+\delta 
}
with $C_0$ depending on $C_0'$, the shape regularity of $\mc{T}_h$, and $\| \rho_u, \rho_a, \rho_w \|_{L^\infty}< \infty$.

Note that the conclusion \eqref{eq:main-error-estm} follows from \eqref{eq:error-energy-eq}, \eqref{eq:Rn-estm}, \eqref{eq:En-estm}, \eqref{eq:E0-estm} by taking $C_1 = \sum_{0\le i\le n} {C_{1,i}}$, $C_2 = \sum_{0\le i\le n} {C_{2,i}}$. Therefore we will devote the rest of proof to prove \eqref{eq:Rn-estm}, \eqref{eq:En-estm}, and \eqref{eq:E0-estm}.

Since the proof of \eqref{eq:Rn-estm} is long, we show \eqref{eq:En-estm} and \eqref{eq:E0-estm} first assuming that \eqref{eq:Rn-estm} is proved. For \eqref{eq:En-estm}, note that  
\algns{
	\mc{E}^{n+1} - \mc{E}^n \le 2 (C_{1,n} (\Delta t)^2 + C_{2,n} h^s)
}
is obtained by \eqref{eq:error-energy-eq} and \eqref{eq:Rn-estm}. Then \eqref{eq:En-estm} follows by induction. For \eqref{eq:E0-estm}, the triangle inequality and $\| \rho_u, \rho_a, \rho_w \|_{L^\infty}< \infty$ give 
\algn{ \label{eq:initial-intm-estm-2}
  \mc{E}^0 &\le \| e_{\ub}^{I,0} \|_{\rho_{u}} +  \|  e_{\vb}^{I,0} \|_{\rho_{a}} +  \|  e_{\wb}^{I,0} \|_{\rho_{w}} + C_0' h^s   \\
  \notag &\le C h^s \| \ub(0), \vb(0), \wb(0) \|_{s} + C_0' h^s , \quad \frac 12 < s \le k+1 + \delta
}
with $C>0$ depending on the shape regularity of $\mc{T}_h$ and $\| \rho_u, \rho_a, \rho_w \|_{L^\infty}$.

Before we prove \eqref{eq:Rn-estm}, let us review some interpolation error estimates from time discretization schemes.
Since the interpolation operator $I_h (= \Pi_h, \bs{P}_h, P_h)$ is independent in time, the time derivative of $I_h \sigma$ is same as $I_h \dot{\sigma}$ as long as they are well-defined pointwisely in time. Then, assuming that a general variable $\sigma \in L^2(t_n, t_{n+1}; L^2)$ is sufficiently regular, we can obtain 
\algn{
  \label{eq:time-intp-error-1} \Delta t \|  \dot \sigma^{n+\frac 12} - \bar{\pd}_t \sigma^{n+\frac 12} \|_0 &= \| \Delta t \dot \sigma^{n+\frac 12} - ( \sigma^{n+1} - \sigma^{n} ) \|_0 \\
  \notag &\le C \Delta t^2 \| \sigma \|_{\dot{W}^{3,1}(t_n, t_{n+1}; L^2)} , \\
  \label{eq:time-intp-error-2} \Delta t \| \bar{\pd}_t \sigma^{n+\frac 12} - \bar{\pd}_t I_h \sigma^{n+\frac 12} \|_0 &= \| \sigma^{n+1} - I_h \sigma^{n+1} - (\sigma^n - I_h \sigma^n) \|_0 \\
  \notag &= \left \| \int_{t_n}^{t_{n+1}} (\dot \sigma (s) - I_h \dot \sigma (s) )\,ds \right \|_0 \\
  \notag &\le C h^s \| \dot{\sigma} \|_{L^{1}(t_n, t_{n+1}; H^s)} 
}
with $s$ satisfying the range conditions in \eqref{eq:s-range-1}, \eqref{eq:s-range-2}, \eqref{eq:s-range-3}.

For the proof of \eqref{eq:Rn-estm}, it suffices to estimate the terms in \eqref{eq:F-terms} by the definition of $R^n$ in \eqref{eq:error-energy-eq}. 

By \eqref{eq:approx-assumption}, \eqref{eq:time-intp-error-1}, \eqref{eq:time-intp-error-2}, and the triangle inequality, and by assuming that the exact solution $\vb$ is sufficiently regular, one can show 
\mltln{ 
  \label{eq:Fv-estm} \Delta t | F_v^n (e_{\vb}^{h,n+\frac 12}) | \\
  \quad \le C \LRp{ (\Delta t)^2 \| \vb \|_{\dot{W}^{3,1}(t_n, t_{n+1}; L^2)} + h^s \| \dot \vb \|_{L^1(t_n, t_{n+1}; H^s)} } \| e_{\vb}^{h,n+\frac 12} \|_{\rho_a} 
}
with $C>0$ depending on $\| \rho_a, \rho_u \|_{L^\infty}$, $\|\rho_a\|_{L^\infty}^{-1}$. 
Noting the identity
\algns{
F_u^n (\bs{P}_h (\rho_u e_\ub^{h,n+\frac 12})) &=  - \LRp{  \bar{\pd}_t \bs{P}_h \ub^{n+\frac 12} - \dot \ub^{n+\frac 12} , \bs{P}_h (\rho_u e_\ub^{h,n+\frac 12})} \\
&= - \LRp{ \bar{\pd}_t \ub^{n+\frac 12} - \dot \ub^{n+\frac 12} , \bs{P}_h (\rho_u e_\ub^{h,n+\frac 12})}
}
and the inequality $\| \bs{P}_h (\rho_u e_\ub^{h,n+\frac 12}) \|_0 \le \| \rho_u e_\ub^{h,n+\frac 12} \|_0 \le \| \sqrt{\rho_u} \|_{L^\infty} \| e_\ub^{h,n+\frac 12} \|_{\rho_u}$, we obtain 
\algn{
  \label{eq:Fu-estm} \Delta t | F_u^n (\bs{P}_h (\rho_u e_\ub^{h,n+\frac 12})) | \le C (\Delta t)^2 \| \ub \|_{\dot{W}^{3,1}(t_n, t_{n+1}; L^2)}  \| e_{\ub}^{h,n+\frac 12} \|_{\rho_u} 
}
with $C>0$ depending on $\| \rho_u \|_{L^\infty}$ by \eqref{eq:time-intp-error-1}. A completely similar argument gives 
\algn{
  \label{eq:Fw-estm} \Delta t | F_w^n (\bs{P}_h (\rho_w e_\wb^{h,n+\frac 12})) | &\le C (\Delta t)^2 \| \wb \|_{\dot{W}^{3,1}(t_n, t_{n+1}; L^2)}  \| e_{\wb}^{h,n+\frac 12} \|_{\rho_w} , \\
  \label{eq:Fq-estm} \Delta t | F_q^n (P_h (\rho_q e_q^{h,n+\frac 12})) | &\le C (\Delta t)^2 \| q \|_{\dot{W}^{3,1}(t_n, t_{n+1}; L^2)}  \| e_q^{h,n+\frac 12} \|_{\rho_q} , \\
  \label{eq:Fr-estm} \Delta t | F_r^n (P_h (\rho_r e_r^{h,n+\frac 12})) | &\le C (\Delta t)^2 \| r \|_{\dot{W}^{3,1}(t_n, t_{n+1}; L^2)}  \| e_{r}^{h,n+\frac 12} \|_{\rho_r} 
}
with $C>0$ depending on $\| \rho_w \|_{L^\infty}$, $\| \rho_q \|_{L^\infty}$, $\| \rho_r \|_{L^\infty}$, respectively. 

Now we only need to estimate the $F_p^n (e_p^{h,n+\frac 12})$-involved term but it needs an additional discussion because the standard approximation theory with $Q_h$ gives only a bound of $O(h^{k+1})$.
To obtain an estimate of $\| e_p^{h,n} \|_{\kappa_a^{-1}}$ with a bound of $O(h^s)$, $\frac 12 < s \le k+1+\delta$, we will use 
\algn{ 
  \label{eq:superconvergence-intp1} \left| \LRp{ \kappa_a^{-1} \bar{\pd}_t e_p^{I,n+\frac 12} , p'} \right|  &= \left| \LRp{ (\kappa_a^{-1} - P_0 \kappa_a^{-1} ) \bar{\pd}_t e_p^{I, n+\frac 12}, p'} \right| \\ 
  \notag &\le Ch \| \kappa_a^{-1} \|_{W_h^{1,\infty}} \| \sqrt{\kappa_a} \|_{L^\infty} \| \bar{\pd}_t e_p^{I, n+\frac 12} \|_0 \| p' \|_{\kappa_a^{-1}} , \\
  \label{eq:superconvergence-intp2} \left| \LRp{ \rho_u e_p^{I, n+\frac 12}, p'} \right| &= \left| \LRp{ (\rho_u - P_0 \rho_u ) e_p^{I, n+\frac 12}, p'} \right| \\ 
  \notag &\le Ch \| \rho_u \|_{W_h^{1,\infty}} \| \sqrt{\kappa_a} \|_{L^\infty} \| e_p^{I,n+\frac 12} \|_{0} \| p' \|_{\kappa_a^{-1}} .
}
By \eqref{eq:approx-assumption}, \eqref{eq:superconvergence-intp1}, \eqref{eq:superconvergence-intp2}, \eqref{eq:time-intp-error-1}, \eqref{eq:time-intp-error-2}, assuming that the exact solutions are sufficiently regular, one can obtain 
\mltln{
  \label{eq:Fp-estm} \Delta t | F_p^n (e_p^{h,n+\frac 12}) | \\
  \le C \LRp{ (\Delta t)^2 \| p \|_{\dot{W}^{3,1}(t_n, t_{n+1}; L^2)} + h^{s} \| \dot p \|_{L^1(t_n, t_{n+1}; H^{s_0})} } \| e_p^{h,n+\frac 12} \|_{\kappa_a^{-1}} 
}
for $0 \le s \le k+1+\delta$ with $C>0$ depending on $\| \kappa_a \|_{L^\infty}$, $\| \rho_u \|_{W_h^{1,\infty}}$, $\| \kappa_a^{-1} \|_{W_h^{1,\infty}}$. Combining \eqref{eq:Fv-estm}, \eqref{eq:Fu-estm}, \eqref{eq:Fw-estm}, \eqref{eq:Fq-estm}, \eqref{eq:Fr-estm}, \eqref{eq:Fp-estm}, and the triangle inequality with the definition of $R^n$, we can obtain \eqref{eq:Rn-estm} with $C_{1,n}$, $C_{2,n}$ with the dependence described in \eqref{eq:C1n-depend}, \eqref{eq:C2n-depend}. 
\end{proof}
\begin{theorem} \label{thm:error-estm}
  Suppose that the assumptions of Proposition~\ref{prop:error-estm} hold and $\delta$ is defined in the same way. Then 
  \algn{ \label{eq:uvw-estm}
    \| \ub^n - \ub_h^n \|_{\rho_u} + \| \vb^n - \vb_h^n \|_{\rho_a} + \| \wb^n - \wb_h^n \|_{\rho_w} \le \mc{E}^n + C h^s \| \ub, \vb, \wb \|_{C^0([0,T]; H^s)}
  }
  with $\frac 12 < s \le k + 1 + \delta$ where $C$ depends on the shape regularity of $\mc{T}_h$ and the degree $k$. Similarly, 
  \algn{ \label{eq:pqr-estm}
    \| p^n - p_h^n \|_{\kappa_a^{-1}} + \| q^n - q_h^n \|_{\rho_q} + \| r^n - r_h^n \|_{\rho_r} \le \mc{E}^n + C h^s \| p, q, r \|_{C^0([0,T]; H^s)}
  }
  with $0 \le s \le k+1$. 
\end{theorem}
\begin{proof}
By the triangle inequality and the definition of $\mc{E}^n$, 
\mltln{ \label{eq:error-intm-estm-3}
  \| \ub^n - \ub_h^n \|_{\rho_u} + \| \vb^n - \vb_h^n \|_{\rho_a} + \| \wb^n - \wb_h^n \|_{\rho_w} \\
  \le \mc{E}^n + \| e_{\ub}^{I,n} \|_{\rho_u} + \| e_{\vb}^{I,n} \|_{\rho_a} + \| e_{\wb}^{I,n} \|_{\rho_w} .
}
By the approximation properties of $\Pi_h$ and $\bs{P}_h$, 
\mltln{ \label{eq:initial-intm-estm-1}
  \| e_{\ub}^{I,n} \|_{\rho_u}  + \| e_{\vb}^{I,n} \|_{\rho_a} + \| e_{\wb}^{I,n} \|_{\rho_w} \\
  \le C h^s \| \ub, \vb, \wb \|_{C^0([0,T]; H^s)} , \quad \frac 12 < s \le k+1 +\delta
}
holds with $C>0$ depending on the shape regularity of $\mc{T}_h$ and $\| \rho_u, \rho_a, \rho_w \|_{L^\infty}$.  Then \eqref{eq:uvw-estm} follows.

Similarly, the triangle inequality gives 
\algns{
    \| p^n - p_h^n \|_{\kappa_a^{-1}} + \| q^n - q_h^n \|_{\rho_q} + \| r^n - r_h^n \|_{\rho_r} \le \mc{E}^n + \| e_{p}^{I,n} \|_{\kappa_a^{-1}} + \| e_{q}^{I,n} \|_{\rho_q} + \| e_{r}^{I,n} \|_{\rho_r} .
}
Then 
\algns{
  \| e_{p}^{I,n} \|_{\kappa_a^{-1}} + \| e_{q}^{I,n} \|_{\rho_q} + \| e_{r}^{I,n} \|_{\rho_r} \le C h^s \| p, q, r \|_{C^0([0,T]; H^s)}, \quad 1 \le s \le k +1
}
holds with $C>0$ depending on the shape regularity of $\mc{T}_h$ and $\| \kappa_a^{-1}, \rho_q, \rho_r \|_{L^\infty}$, so \eqref{eq:pqr-estm} follows. 
\end{proof}

\subsection{Error analysis for post-processed solutions} \label{subsec:post-processing}
If $\Vb_h$ is a BDM element, then $\delta = 1$ and the optimal convergence rate in \eqref{eq:pqr-estm} is one order lower than the one in \eqref{eq:uvw-estm}. This lower convergence rate can be circumvented by a local post-processing which we introduce below. 

Throughout this subsection we assume that the exact solutions are sufficiently regular and we will not concern about low regularity of exact solutions. 
In our local post-processing, we first find $\{p_h^{n, *}\}_{n=0}^{N-1}$, a new numerical solution approximating $p(t_n + \Delta t/2)$, not $p(t_n)$. The goal is to show that $\| p_h^{n, *} - p(t_n + \Delta t/2) \|_0$ can have $O(h^{k+2})$ convergence rate.

To define the local post-processing let us define 
\algns{
  Q_h^* &= \{ q \in L^2(\Omega)\,:\, q|_K \in \mc{P}_{k+2}(K), \quad K \in \mc{T}_h \} .
}
We define $p_h^{n,*} \in Q_h^*$ as 
\algn{
\label{eq:pstar-def-1} \int_K p_h^{n,*} \, dx &= \int_K p_h^{n+\frac 12} \, dx, \\
\label{eq:pstar-def-2} \LRp{\grad p_h^{n,*}, \grad p' }_K &= - \LRp{\rho_a \bar{\pd}_t \vb^{n+\frac 12} , \grad p'}_K - \LRp{\rho_u \ub_h^{n+\frac 12} , \grad p'}_K  \\
  \notag &\quad + \LRp{\fb^{n+\frac 12} , \grad p'}_K , \quad \forall p' \in Q_h^*
}
for all $K \in \mc{T}_h$. 
Note that this post-processing is solving a system with a block diagonal matrix such that the size of each matrix block is the number of DOFs of $Q_h^*$ on one simplex $K$. Therefore, the computational costs are negligibly small compared to the computational costs of the original linear system. 
\begin{lemma}
  Suppose that the assumptions of Proposition~\ref{prop:error-estm} hold and $\Vb_h$ is a BDM element. If $\{p_h^{n,*}\}_{n=0}^{N-1}$ is defined as in \eqref{eq:pstar-def-1}, \eqref{eq:pstar-def-2}, and the exact solution $(\vb, p, \ub, \wb, q, r)$ is sufficiently regular, then 
\algns{
  \| p(t_n + \Delta t/2) - p_h^{*,n+\frac 12} \|_0 \le C ((\Delta t)^2 + h^{k+2} ) 
}
with a constant $C>0$ which depends on the constants $C_0$, $C_1$, $C_2$ in Proposition~\ref{prop:error-estm}, $\| \ub \|_{C^0([0,T]; H^{k+1})}$, $\| p \|_{C^0([0,T]; H^{k+2})}$, and $\| p \|_{\dot{W}^{2,\infty}(t_n, t_{n+1}; L^2)}$. 
\end{lemma}
\begin{proof}
Recall that $p^{n+\frac 12} = \frac 12 (p^n + p^{n+1}) = \frac 12 (p(t_n) + p(t_{n+1}))$. 
We first note that 
\algns{
	\| p(t_n + \Delta t/2) - p^{n+\frac 12} \|_0 \le C(\Delta t)^2 \| p \|_{\dot{W}^{2,\infty}(t_n, t_{n+1}; L^2)}
}
by an argument with the Taylor expansion. By the triangle inequality it suffices to estimate $\| p^{n+\frac 12} - p_h^{*,n+\frac 12} \|_0$. 

To show an error estimate of $\| p^{n+\frac 12} - p_h^{n,*} \|_0$ consider the error equation
\mltlns{
  \LRp{ \grad (p^{n+\frac 12} - p_h^{n,*}), \grad p'} \\
   = - \LRp{ \rho_a (\dot \vb^{n+\frac 12} - \bar{\pd}_t \vb_h^{n+\frac 12}), \grad p'} - \LRp{\rho_u (\ub^{n+\frac 12} - \ub_h^{n+\frac 12}) , \grad p'} \quad \forall p' \in Q_h^*
}
from the definition of $p_h^{n,*}$ and \eqref{eq:original-eq-1}.

For $P_h^*$, the $L^2$ projection to $Q_h^*$, we can rewrite this equation as 
\algns{
  &\LRp{ \grad (P_h^* p^{n+\frac 12} - p_h^{n,*}), \grad p'} \\
  &= - \LRp{ \grad ( p^{n+\frac 12} - P_h^* p^{n+\frac 12}), \grad p'} \\
  &\quad - \LRp{ \rho_a (\dot \vb^{n+\frac 12} - \bar{\pd}_t \vb_h^{n+\frac 12}), \grad p'} - \LRp{\rho_u (\ub^{n+\frac 12} - \ub_h^{n+\frac 12}) , \grad p'} .
}
If we take $p' = P_h^* p^{n+\frac 12} - p_h^{n,*}$ and use the Cauchy--Schwarz inequality, then we get
\algn{ 
  \notag &\| \grad (P_h^* p^{n+\frac 12} - p_h^{n,*}) \|_0 \\
  \label{eq:grad-pstar-estm} &\le  \| \grad ( p^{n+\frac 12} - P_h^* p^{n+\frac 12}) \|_0 + C \LRp{ \| \dot \vb^{n+\frac 12} - \bar{\pd}_t \vb_h^{n+\frac 12} \|_0 + \| \ub^{n+\frac 12} - \ub_h^{n+\frac 12} \|_{\rho_u} } \\
  \notag &=: I_1 + I_2 + I_3 
}
with $C>0$ depending on $\rho_a$ and $\Omega_{\rho}$.

To estimate $I_1$, assuming that $p$ is sufficiently regular, we use the Bramble--Hilbert lemma and get 
\algn{ \label{eq:I1-estm}
  \| \grad ( p^{n+\frac 12} - P_h^* p^{n+\frac 12}) \|_{0,K} \le Ch_K^{k+1} \| p^n, p^{n+1} \|_{k+2,K} , \qquad \forall K \in \mc{T}_h .
}
An estimate of $I_3$ is obtained by Theorem~\ref{thm:error-estm} as 
\algn{ \label{eq:I3-estm}
  \| \ub^{n+\frac 12} - \ub_h^{n+\frac 12} \|_{\rho_u} \le C( (\Delta t)^2 + h^{k+2} )
}
under the assumption that $\ub$ is sufficiently regular.

We estimate $I_2$ by estimating 
\algn{
  I_2^a := \| \dot \vb^{n+\frac 12} - \bar{\pd}_t \vb^{n+\frac 12} \|_0, \quad I_2^b := \| \bar{\pd}_t e_{\vb}^{I, n+\frac 12} \|_0, \quad I_2^c := \| \bar{\pd}_t e_{\vb}^{h, n+\frac 12} \|_0 .
}
By \eqref{eq:time-intp-error-1} and \eqref{eq:time-intp-error-2}, 
\algn{ 
  \label{eq:I2a-estm}  I_2^a &\le C (\Delta t)^2 \| \vb \|_{\dot{W}^{3,\infty}(t_n, t_{n+1}; L^2)} , \\
  \label{eq:I2b-estm}  I_2^b &\le Ch^{k+1} \| \dot \vb \|_{L^\infty (t_n, t_{n+1}; H^{k+1})} .
}

The estimate of $I_2^c$ is more technical. First, note that it is enough to estimate $\| \bar{\pd}_t e_{\vb}^{h,n+\frac 12} \|_{\rho_a}$ since $\rho_a>0$ is uniformly positive. It is known (cf. \cite{Arnold-Falk-Winther:1997,Arnold-Falk-Winther:2000}) that there is a decomposition $\bar{\pd}_t e_{\vb}^{h,n+\frac 12} = \vb_0 + \vb_1$ with $\vb_0, \vb_1 \in \Vb_h$ such that 
\algns{
  \div \vb_0 = 0, \quad \LRp{\rho_a \vb_0, \vb_1 } = 0, \quad \| \div \vb_1 \|_0 \le C \| \vb_1 \|_0 
}
with $C>0$ independent of $h$. 
Using this decomposition, we can rewrite \eqref{eq:full-error-eq-1} as 
\algns{
  \LRp{\rho_{a} (\vb_0 + \vb_1) , \vb'} - \LRp{ e_{p}^{h,n+\frac 12}, \div \vb'} + \LRp{ \rho_{u} e_{\ub}^{h,n+\frac 12}, \vb'} &= F_v^n (\vb') .
}
If $\vb' = \vb_1$, then 
\algns{
  \| \vb_1 \|_{\rho_a}^2 &\le (C \| e_p^{h, n+\frac 12} \|_0 + \| e_{\ub}^{h, n+\frac 12} \|_{\rho_u}) \| \vb_1 \|_0 + | F_v^n(\vb_1) | .
}
Recall that $\| e_p^{h, n+\frac 12} \|_0$, $\| e_{\ub}^{h, n+\frac 12} \|_{\rho_u}$ are estimated in Proposition~\ref{prop:error-estm} and $| F_v^n(\vb_1) |$ is estimated by \eqref{eq:Fv-estm}. As a consequence, 
\algns{
  \| \vb_1 \|_0 \le C (h^{k+1} + (\Delta t)^2)
}
holds with $C>0$ depending on $\| \vb \|_{\dot{W}^{3,\infty}(t_n, t_{n+1}; L^2)}$, $\| \dot \vb \|_{L^\infty(t_n, t_{n+1}; H^{k+1})}$, and the constants $C_0$, $C_1$, $C_2$ in Proposition~\ref{prop:error-estm}.  
If $\vb' = \vb_0$, then we get 
\algns{ 
  \| \vb_0 \|_{\rho_a}^2 &\le \| e_{\ub}^{h, n+\frac 12} \|_0 \| \vb_0 \|_0 + | F_v^n(\vb_0) | .
}
An estimate of $\| \vb_0 \|_0$ can be obtained by a completely similar argument for the estimate of $\| \vb_1 \|_0$. Therefore, by combining the estimates of $\| \vb_0 \|_0$ and $\| \vb_1 \|_0$, we have 
\algn{ \label{eq:I2c-estm}
  I_2 \le C((\Delta t)^2 + h^{k+1})
}
with $C>0$ depending on $\|  \vb \|_{\dot{W}^{3,\infty}(t_n, t_{n+1}; L^2)}$, $\| \dot \vb \|_{L^\infty(t_n, t_{n+1}; H^{k+1})}$, and the constants $C_0$, $C_1$, $C_2$ in Proposition~\ref{prop:error-estm}.

By combining \eqref{eq:grad-pstar-estm}, \eqref{eq:I1-estm}, \eqref{eq:I3-estm}, \eqref{eq:I2a-estm}, \eqref{eq:I2b-estm}, \eqref{eq:I2c-estm}, we obtained
\algns{
  \| \grad (P_h^* p^{n+\frac 12} - p_h^{*,n + \frac 12}) \|_0 \le C ((\Delta t)^2 + h^{k+1} ) .
}
An element-wise Poincar\'{e} inequality and the above estimate give
\algns{
  \| P_h^* p^{n+\frac 12} - p_h^{n,* } \|_0 \le Ch \| \grad (P_h^* p^{n+\frac 12} - p_h^{n,*}) \|_0 \le Ch ((\Delta t)^2 + h^{k+1} ) .
}
From this we can obtain 
\algns{
\| p^{n+\frac 12} - p_h^{n,*} \|_0 &\le \| p^{n+\frac 12} - P_h^* p^{n,*} \|_0 + \| P_h^* p^{n+\frac 12} - p_h^{n,*} \|_0 \\
&\le C h^{k+2} \| p^{n+\frac 12} \|_{k+2} + Ch((\Delta t)^2 + h^{k+1} ),
}
which is the desired estimate.
\end{proof}

\section{Numerical results} \label{sec:numerical}

In this section we present the results of numerical experiments to illustrate the validity of our theoretical analysis.\footnote{The datasets generated during and/or analyzed during the current study are available from the corresponding author on reasonable request.} 

In the first set of experiments we use a manufactured solution and show the convergence rates of errors with various finite element discretizations. Specifically, let $\Omega = [0,1]\times [0,1]$ with the subdomain $\Omega_0 = [3/8, 5/8] \times [0,1]$. We set $\rho_a = \kappa_a = \omega_{\rho} = \omega_{\kappa} = 1$, $\gamma=0$ on $\Omega$ whereas $\Omega_{\rho}$, $\Omega_{\kappa}$ are defined as 
\algns{
	\Omega_{\rho} = \Omega_{\kappa} = 
	\begin{cases}
		1	\text{ on } \Omega_0 \\
		0	\text{ on } \Omega \setminus \overline{\Omega}_0
	\end{cases}
}
A manufactured solution is constructed with 
\algn{ \label{eq:exact-solution}
	\wb(x,y) = \pmat{(1+\sin t) (x^2 y + x y^2) \\ \cos(2t) (x+y + \cos x) }, \qquad r(x,y) = \cos(3t) xy ,
}
and the other functions $\ub(x,y)$, $\vb(x,y)$, $q(x,y)$, $p(x,y)$, $\fb(x,y)$, $g(x,y)$ are defined by \eqref{eq:original-eqs}.

\begin{table}[!h] 
\setlength{\tabcolsep}{4.8pt} 
\caption{Errors and convergence rates with $\Vb_h(K) \times Q_h(K) \times \Wb_h(K)= {\rm BDM}_1 (K) \times \mc{P}_0 (K) \times \mc{P}_1(K; \R^d)$ for the exact solution in \eqref{eq:exact-solution}.  }
\label{table:BDM1}
\begin{tabular} 
{>{\small}c >{\small}c >{\small}c >{\small}c >{\small}c >{\small}c >{\small}c >{\small}c >{\small}c } 
 \hline 
\multirow{2}{*}{$\frac 1h$} & \multicolumn{2}{>{\small}c}{$\|\vb - \vb_h \|_0$} & \multicolumn{2}{>{\small}c}{$\| \ub - \ub_h \|_0$} & \multicolumn{2}{>{\small}c}{$\| \wb - \wb_h \|_0$} & \multicolumn{2}{>{\small}c}{$\| p - p_h \|_0$} \\ 
 & error & rate & error & rate & error & rate & error & rate \\ \hline 
8 & 5.53e-03 &  -- & 9.51e-03 &  -- & 3.65e-03 &  -- & 1.65e-01 &  -- \\  
16 & 1.34e-03 & 2.04 & 2.39e-03 & 1.99 & 9.15e-04 & 2.00 & 8.26e-02 & 1.00 \\  
32 & 3.49e-04 & 1.94 & 5.98e-04 & 2.00 & 2.29e-04 & 2.00 & 4.13e-02 & 1.00 \\  
64 & 8.42e-05 & 2.05 & 1.49e-04 & 2.00 & 5.72e-05 & 2.00 & 2.06e-02 & 1.00 \\  \hline 
\multirow{2}{*}{$\frac 1h$} & \multicolumn{2}{>{\small}c}{$\| p - p_h^* \|_0$} & \multicolumn{2}{>{\small}c}{$\| q - q_h \|_0$} & \multicolumn{2}{>{\small}c}{$\| r - r_h \|_0$}\\ 
 & error & rate & error & rate & error & rate \\ \hline 
8 & 5.53e-03 &  -- & 9.51e-03 &  -- & 3.65e-03 &  -- \\  
16 & 1.34e-03 & 2.04 & 2.39e-03 & 1.99 & 9.15e-04 & 2.00 \\  
32 & 3.49e-04 & 1.94 & 5.98e-04 & 2.00 & 2.29e-04 & 2.00 \\  
64 & 8.42e-05 & 2.05 & 1.49e-04 & 2.00 & 5.72e-05 & 2.00 \\  
 \hline
\end{tabular} 
\end{table}

\begin{table}[!h] 
\setlength{\tabcolsep}{4.8pt} 
\caption{Errors and convergence rates with $\Vb_h(K) \times Q_h(K) \times \Wb_h(K)= {\rm RTN}_0 (K) \times \mc{P}_0 (K) \times \mc{P}_0(K; \R^d)$ for the exact solution in \eqref{eq:exact-solution}.  }
\label{table:RTN1}
\begin{tabular} 
{>{\small}c >{\small}c >{\small}c >{\small}c >{\small}c >{\small}c >{\small}c >{\small}c >{\small}c } 
 \hline 
\multirow{2}{*}{$\frac 1h$} & \multicolumn{2}{>{\small}c}{$\|\vb - \vb_h \|_0$} & \multicolumn{2}{>{\small}c}{$\| \ub - \ub_h \|_0$} & \multicolumn{2}{>{\small}c}{$\| \wb - \wb_h \|_0$} & \multicolumn{2}{>{\small}c}{$\| p - p_h \|_0$} \\ 
 & error & rate & error & rate & error & rate & error & rate \\ \hline 
8 & 1.78e-01 &  -- & 7.50e-02 &  -- & 7.61e-02 &  -- & 1.65e-01 &  -- \\  
16 & 8.86e-02 & 1.01 & 3.74e-02 & 1.00 & 3.81e-02 & 1.00 & 8.26e-02 & 1.00 \\  
32 & 4.43e-02 & 1.00 & 1.87e-02 & 1.00 & 1.90e-02 & 1.00 & 4.13e-02 & 1.00 \\  
64 & 2.21e-02 & 1.00 & 9.34e-03 & 1.00 & 9.52e-03 & 1.00 & 2.06e-02 & 1.00 \\  \hline 
\multirow{2}{*}{$\frac 1h$} & \multicolumn{2}{>{\small}c}{$\| p - p_h^* \|_0$} & \multicolumn{2}{>{\small}c}{$\| q - q_h \|_0$} & \multicolumn{2}{>{\small}c}{$\| r - r_h \|_0$}\\ 
 & error & rate & error & rate & error & rate \\ \hline 
8 & 1.78e-01 &  -- & 7.50e-02 &  -- & 7.61e-02 &  -- \\  
16 & 8.86e-02 & 1.01 & 3.74e-02 & 1.00 & 3.81e-02 & 1.00 \\  
32 & 4.43e-02 & 1.00 & 1.87e-02 & 1.00 & 1.90e-02 & 1.00 \\  
64 & 2.21e-02 & 1.00 & 9.34e-03 & 1.00 & 9.52e-03 & 1.00 \\  
\hline
\end{tabular} 
\end{table}

\begin{table}[!h] 
\setlength{\tabcolsep}{4.8pt} 
\caption{Errors and convergence rates with $\Vb_h(K) \times Q_h(K) \times \Wb_h(K)= {\rm BDM}_2 (K) \times \mc{P}_1 (K) \times \mc{P}_2(K; \R^d)$ for the exact solution in \eqref{eq:exact-solution}.  }
\label{table:BDM2}
\begin{tabular} 
{>{\small}c >{\small}c >{\small}c >{\small}c >{\small}c >{\small}c >{\small}c >{\small}c >{\small}c } 
 \hline 
\multirow{2}{*}{$\frac 1h$} & \multicolumn{2}{>{\small}c}{$\|\vb - \vb_h \|_0$} & \multicolumn{2}{>{\small}c}{$\| \ub - \ub_h \|_0$} & \multicolumn{2}{>{\small}c}{$\| \wb - \wb_h \|_0$} & \multicolumn{2}{>{\small}c}{$\| p - p_h \|_0$} \\ 
 & error & rate & error & rate & error & rate & error & rate \\ \hline 
8 & 1.18e-04 &  -- & 1.43e-04 &  -- & 5.91e-05 &  -- & 4.03e-03 &  -- \\  
16 & 1.08e-05 & 3.44 & 9.90e-06 & 3.85 & 6.22e-06 & 3.25 & 1.01e-03 & 2.00 \\  
32 & 1.38e-06 & 2.97 & 8.13e-07 & 3.61 & 7.37e-07 & 3.08 & 2.52e-04 & 2.00 \\  
64 & 1.74e-07 & 2.98 & 8.33e-08 & 3.29 & 9.08e-08 & 3.02 & 6.30e-05 & 2.00 \\  \hline 
\multirow{2}{*}{$\frac 1h$} & \multicolumn{2}{>{\small}c}{$\| p - p_h^* \|_0$} & \multicolumn{2}{>{\small}c}{$\| q - q_h \|_0$} & \multicolumn{2}{>{\small}c}{$\| r - r_h \|_0$}\\ 
 & error & rate & error & rate & error & rate \\ \hline 
8 & 1.18e-04 &  -- & 1.43e-04 &  -- & 5.91e-05 &  -- \\  
16 & 1.08e-05 & 3.44 & 9.90e-06 & 3.85 & 6.22e-06 & 3.25 \\  
32 & 1.38e-06 & 2.97 & 8.13e-07 & 3.61 & 7.37e-07 & 3.08 \\  
64 & 1.74e-07 & 2.98 & 8.33e-08 & 3.29 & 9.08e-08 & 3.02 \\  
\hline
\end{tabular} 
\end{table}

\begin{table}[!h] 
\setlength{\tabcolsep}{4.8pt} 
\caption{Errors and convergence rates with $\Vb_h(K) \times Q_h(K) \times \Wb_h(K)= {\rm RTN}_1 (K) \times \mc{P}_1 (K) \times \mc{P}_1(K; \R^d)$ for the exact solution in \eqref{eq:exact-solution}.  }
\label{table:RTN2}
\begin{tabular} 
{>{\small}c >{\small}c >{\small}c >{\small}c >{\small}c >{\small}c >{\small}c >{\small}c >{\small}c } 
 \hline 
\multirow{2}{*}{$\frac 1h$} & \multicolumn{2}{>{\small}c}{$\|\vb - \vb_h \|_0$} & \multicolumn{2}{>{\small}c}{$\| \ub - \ub_h \|_0$} & \multicolumn{2}{>{\small}c}{$\| \wb - \wb_h \|_0$} & \multicolumn{2}{>{\small}c}{$\| p - p_h \|_0$} \\ 
 & error & rate & error & rate & error & rate & error & rate \\ \hline 
8 & 2.97e-03 &  -- & 2.23e-03 &  -- & 2.75e-03 &  -- & 4.10e-03 &  -- \\  
16 & 6.93e-04 & 2.10 & 5.58e-04 & 2.00 & 6.88e-04 & 2.00 & 1.01e-03 & 2.02 \\  
32 & 2.01e-04 & 1.79 & 1.40e-04 & 2.00 & 1.72e-04 & 2.00 & 2.55e-04 & 1.99 \\  
64 & 4.56e-05 & 2.14 & 3.49e-05 & 2.00 & 4.30e-05 & 2.00 & 6.40e-05 & 1.99 \\  \hline 
\multirow{2}{*}{$\frac 1h$} & \multicolumn{2}{>{\small}c}{$\| p - p_h^* \|_0$} & \multicolumn{2}{>{\small}c}{$\| q - q_h \|_0$} & \multicolumn{2}{>{\small}c}{$\| r - r_h \|_0$}\\ 
 & error & rate & error & rate & error & rate \\ \hline 
8 & 2.97e-03 &  -- & 2.23e-03 &  -- & 2.75e-03 &  -- \\  
16 & 6.93e-04 & 2.10 & 5.58e-04 & 2.00 & 6.88e-04 & 2.00 \\  
32 & 2.01e-04 & 1.79 & 1.40e-04 & 2.00 & 1.72e-04 & 2.00 \\  
64 & 4.56e-05 & 2.14 & 3.49e-05 & 2.00 & 4.30e-05 & 2.00 \\ 
\hline 
\end{tabular} 
\end{table} 

We consider 4 different spatial discretizations such that the local finite element spaces $\Vb_h(K) \times Q_h(K) \times \Wb_h(K)$ are
\algn{
\label{eq:space-disc1}  {\rm BDM}_1 (K) \times \mc{P}_0 (K) \times \mc{P}_1(K; \R^d), \qquad 	{\rm RTN}_0 (K) \times \mc{P}_0 (K) \times \mc{P}_0(K; \R^d), \\
\label{eq:space-disc2} 	{\rm BDM}_2 (K) \times \mc{P}_1 (K) \times \mc{P}_2(K; \R^d), \qquad 	{\rm RTN}_1 (K) \times \mc{P}_1 (K) \times \mc{P}_1(K; \R^d) .
}

For triangulation we use the structured meshes obtained by the bisection of uniform $N \times N$ squares of $\Omega$ for $N = 8, 16, 32, 64$. Then the maximum mesh size is a multiple of $h = 1/N$ with a uniform constant independent of $N$. 
We use the Crank--Nicolson scheme for time discretization with $\Delta t = h$ for \eqref{eq:space-disc1} and with $\Delta t = h^2$ for \eqref{eq:space-disc2} in order to see optimal convergence rates of spatial discretization errors. For these four different discretization schemes the errors and convergence rates are presented in Tables~
\ref{table:BDM1}--\ref{table:RTN2}. 

In these experiments the errors of $\| \ub - \ub_h \|_0$, $\| \vb - \vb_h \|_0$, $\| \wb - \wb_h \|_0$, $\| p - p_h \|_0$, $\| q - q_h \|_0$, $\| r- r_h \|_0$ are computed at $T = 0.25$ whereas the error $\| p - p_h^* \|_0$ is computed at $T - \frac 12 \Delta t$. Although we used the weighted norms $\| \cdot \|_{\rho_u}$, $\| \cdot \|_{\rho_w}$, $\| \cdot \|_{\rho_q}$, $\| \cdot \|_{\rho_r}$ in our error analysis for the errors of $\ub$, $\wb$, $q$, $r$, here we compute the standard $L^2$ norms for those errors. Since the $L^2$ norms are the upper bounds of the weighted norms, the optimal convergence rates of the $L^2$ norm errors are the results stronger than the optimal convergence rates of the weighted norm errors. In all of these experiments we can see the convergence rates which are expected in our error analysis.

\begin{figure}[h!] 
\centering
\includegraphics[width=60mm, height=60mm]{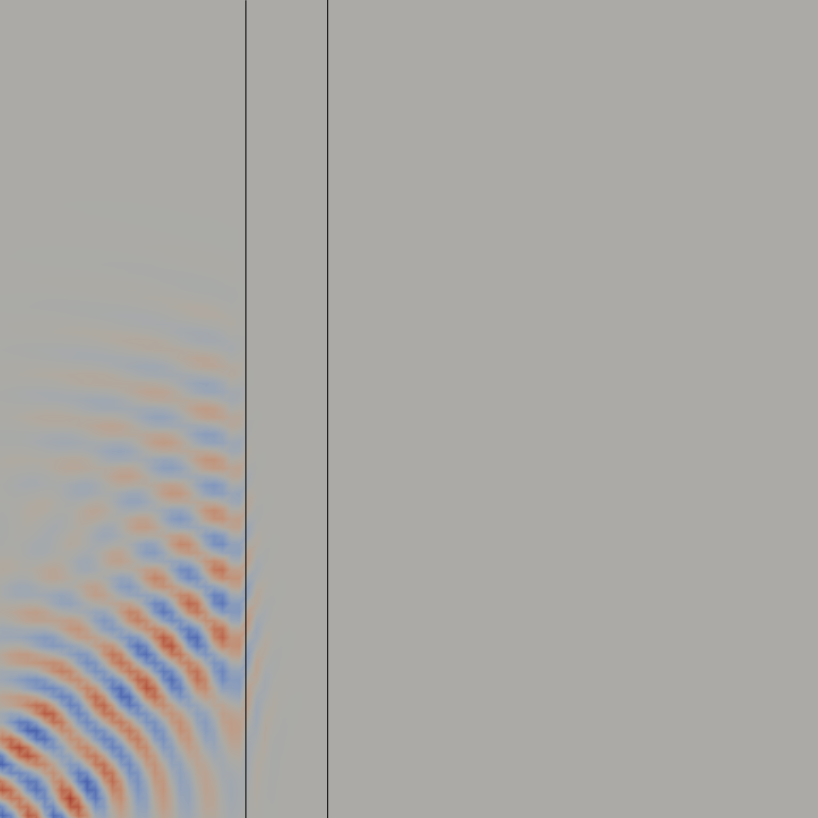} \quad 
\includegraphics[width=60mm, height=60mm]{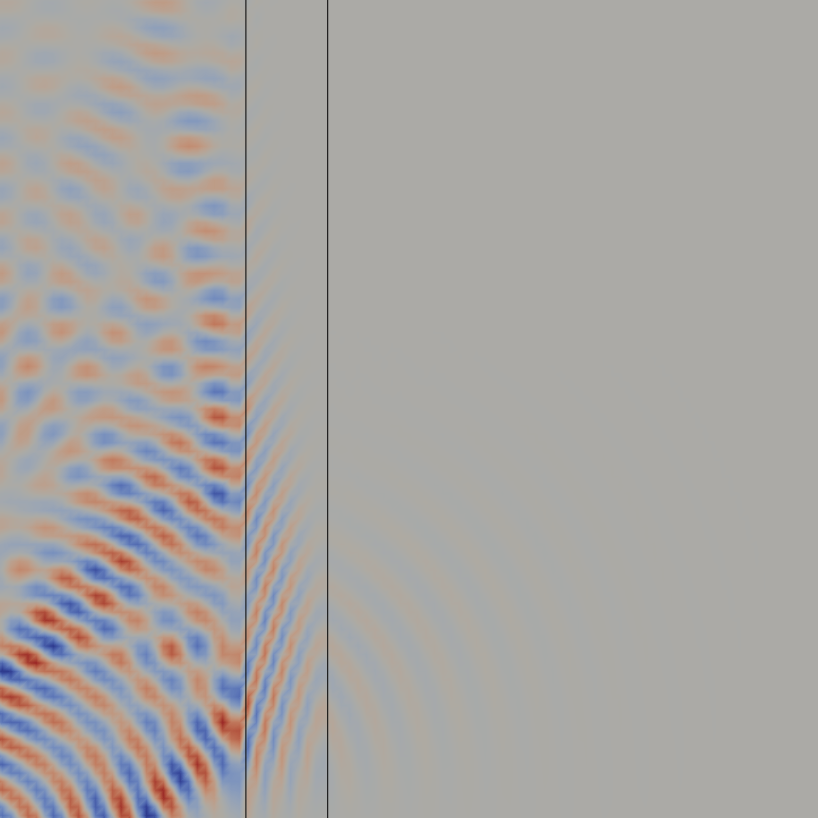} 
\caption{Wave propagation with $p_D$ in \eqref{eq:left-bottom-source}, $\mu_f = 18$, at $t = 0.2, 0.4$ }
\label{fig:wave-mu-18}
\end{figure} 
\begin{figure} 
\centering
\includegraphics[width=60mm, height=60mm]{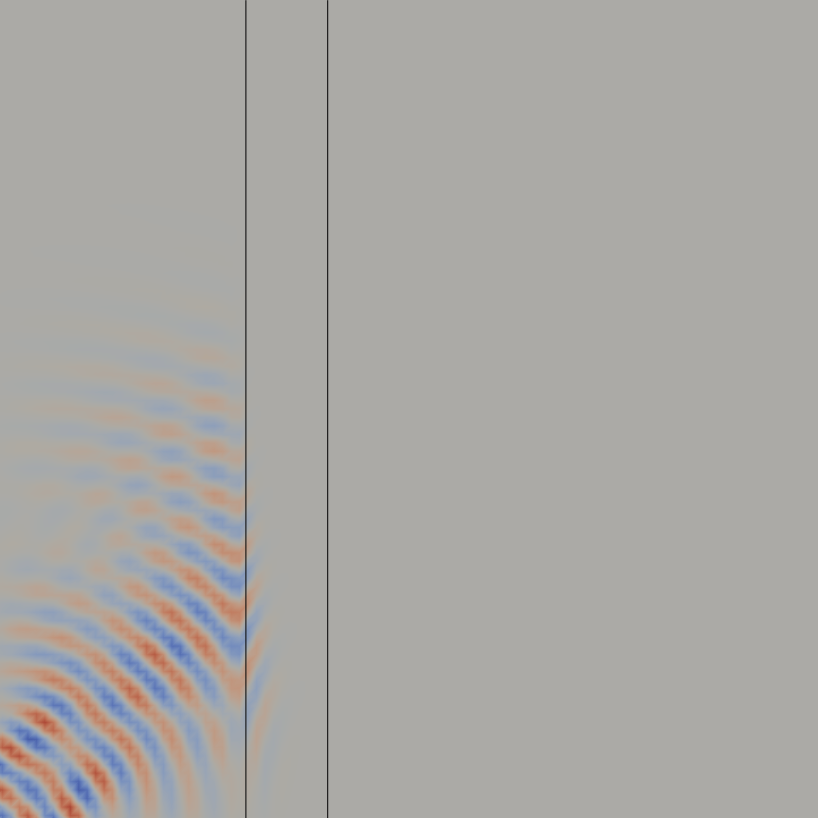} \quad 
\includegraphics[width=60mm, height=60mm]{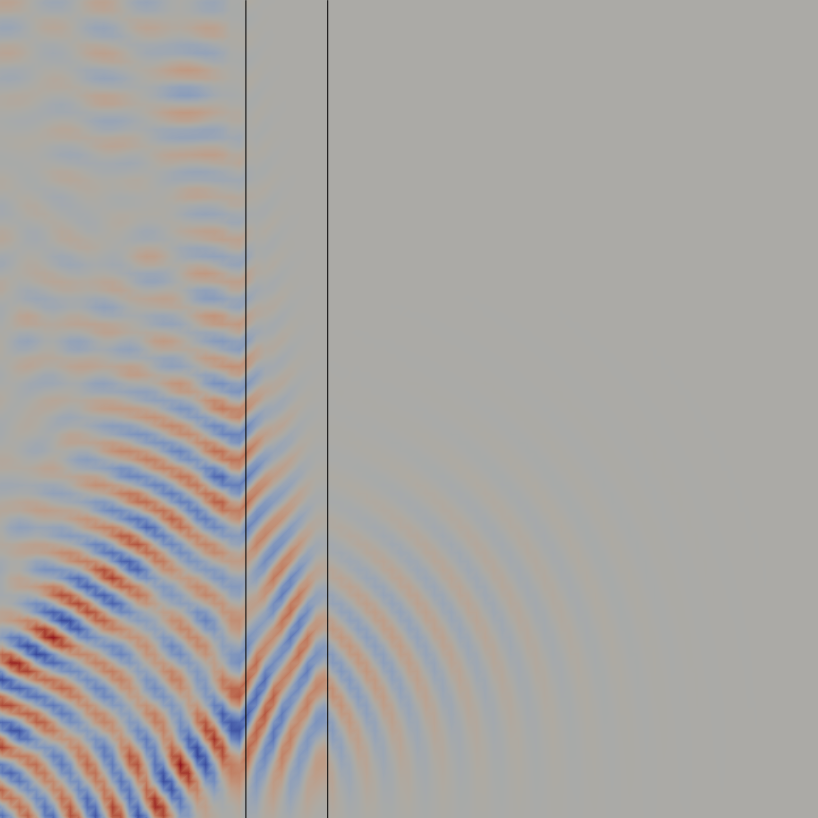} 
\caption{Wave propagation with $p_D$ in \eqref{eq:left-bottom-source}, $\mu_f = 19$, at $t = 0.2, 0.4$}
\label{fig:wave-mu-19}
\end{figure} 
\begin{figure} 
\centering
\includegraphics[width=60mm, height=60mm]{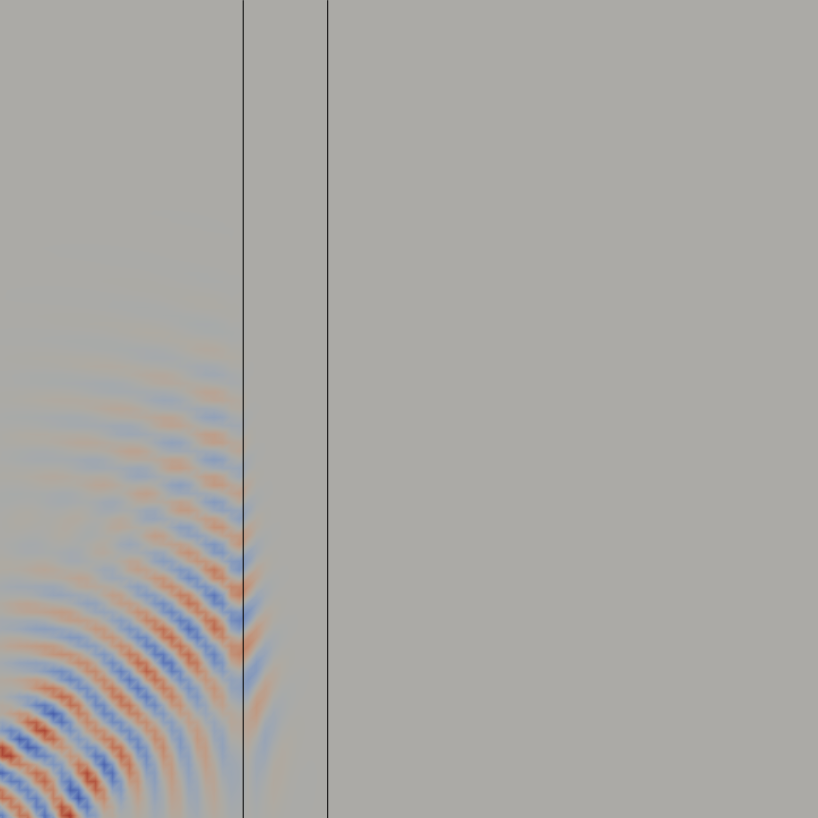} \quad 
\includegraphics[width=60mm, height=60mm]{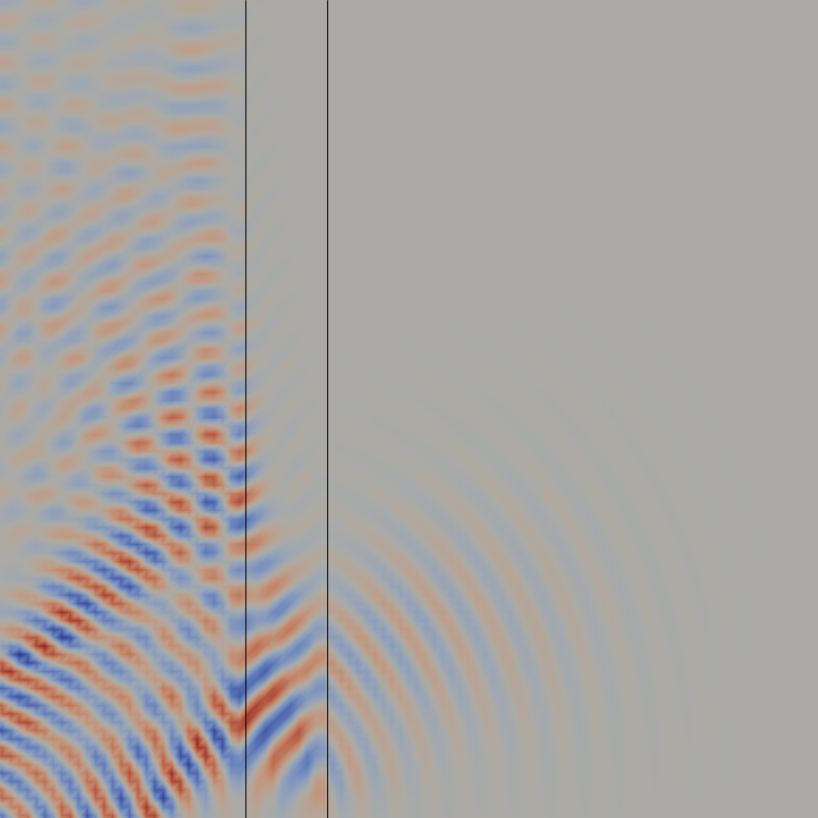} 
\caption{Wave propagation with $p_D$ in \eqref{eq:left-bottom-source}, $\mu_f = 20$, at $t = 0.2, 0.4$}
\label{fig:wave-mu-20}
\end{figure} 

In the second set of experiments we set $\Omega = [0,2]\times [0,2]$ with $\Omega_0 = [3/5,4/5]\times [0,2]$. 
We assume that the parameters are given as 
\algns{
\Omega_{\rho} = \Omega_{\kappa} = 
\begin{cases} 
80 \quad \text{ on }\Omega_0 \\
0 \quad \text{ on } \Omega \setminus \Omega_0 
\end{cases} , 
\qquad \omega_{\rho} = \omega_{\kappa} = 40 \quad \text{ on } \Omega,  
}
so the medium is a metamaterial on $\Omega_0$ but is a conventional material on $\Omega \setminus \Omega_0$. 

We remark that the choices of these parameters are made without consideration of physical ranges of parameter values. 
We also set 
\algn{ \label{eq:left-bottom-source}
p_D(t,x,y) = 
\begin{cases}
	10 \sin (\mu_f \pi (x+y-10 t) & \text{ if } t > x + y \text{ and }x < 3/5, \\
	0 & \text{ otherwise}
\end{cases} 
}
where $\mu_f$ is a constant. In the following experiments we impose the boundary condition \eqref{eq:general-bc} with $p_D$ in the above and $\Gamma_D = \pd \Omega$. Speaking more intuitively, this condition gives an incoming wave propagation from the bottom-left corner of $\Omega$ with frequency $5\mu_f$. 
\begin{figure}[h!] 
\centering
\includegraphics[width=60mm, height=60mm]{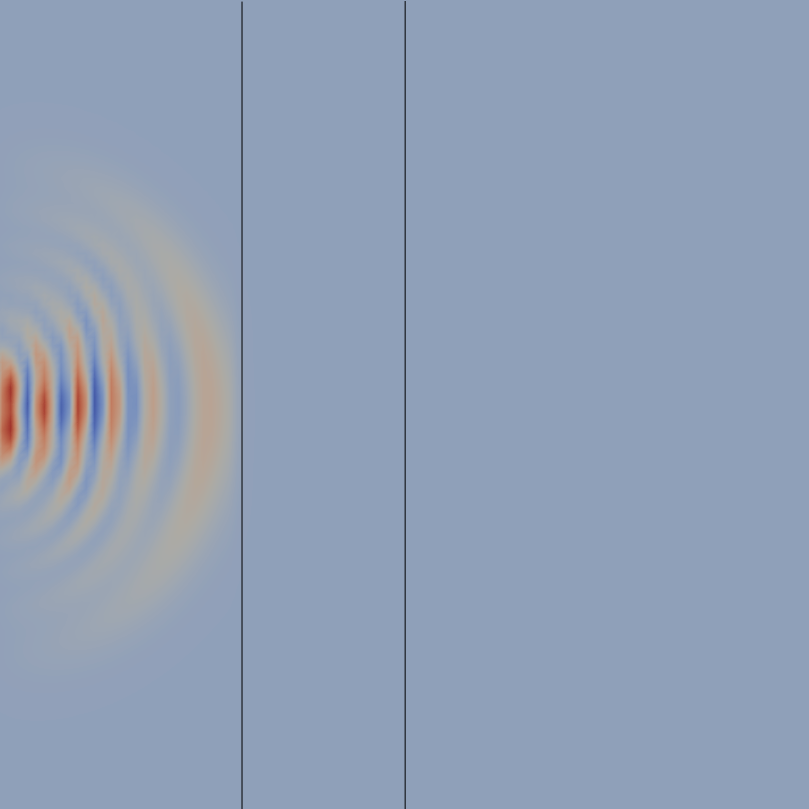} \quad 
\includegraphics[width=60mm, height=60mm]{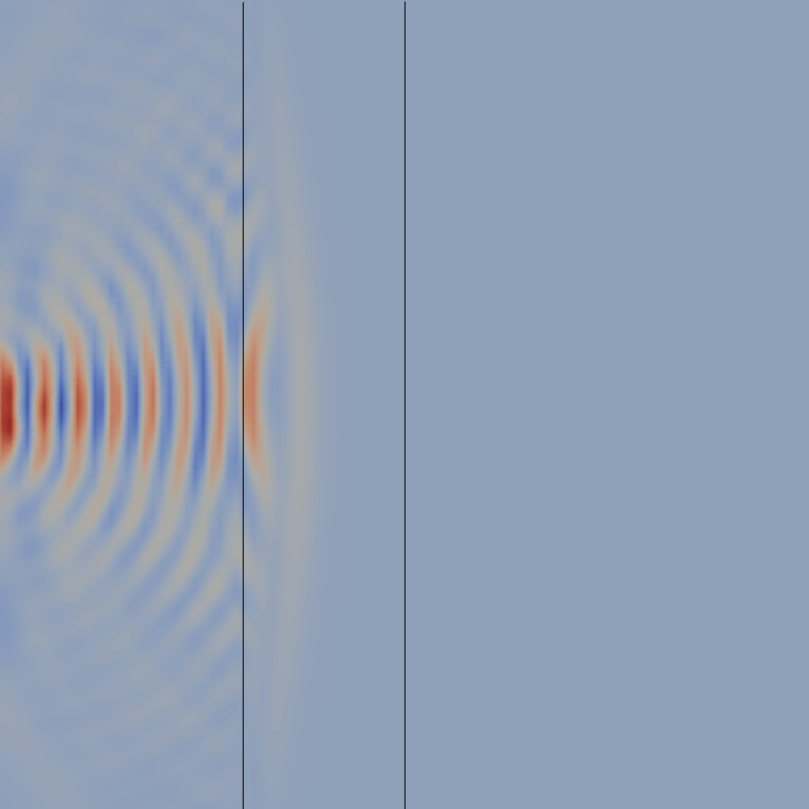} \\ \bigskip
\includegraphics[width=60mm, height=60mm]{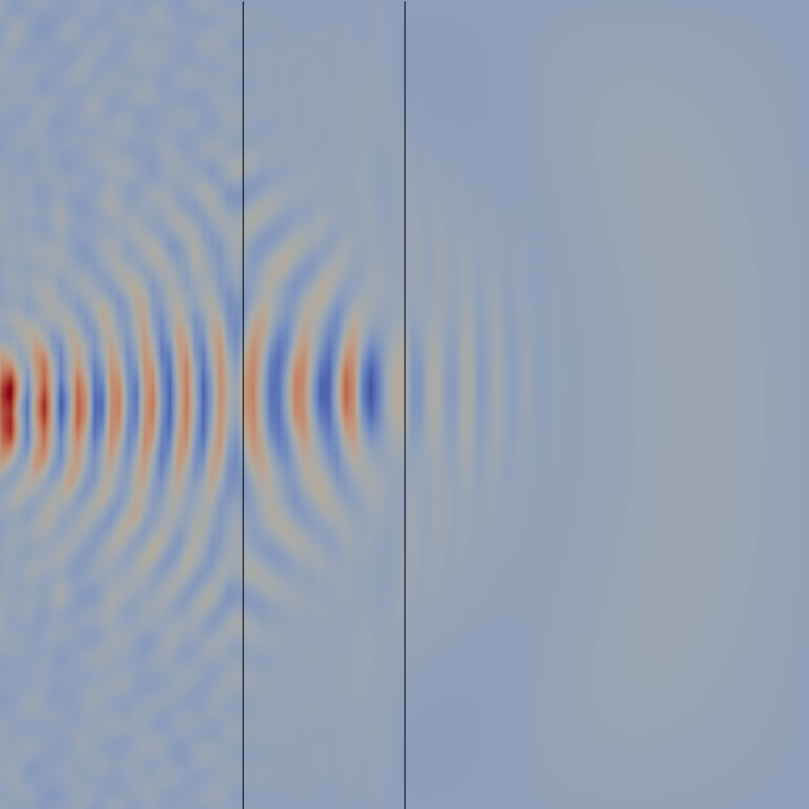} \quad 
\includegraphics[width=60mm, height=60mm]{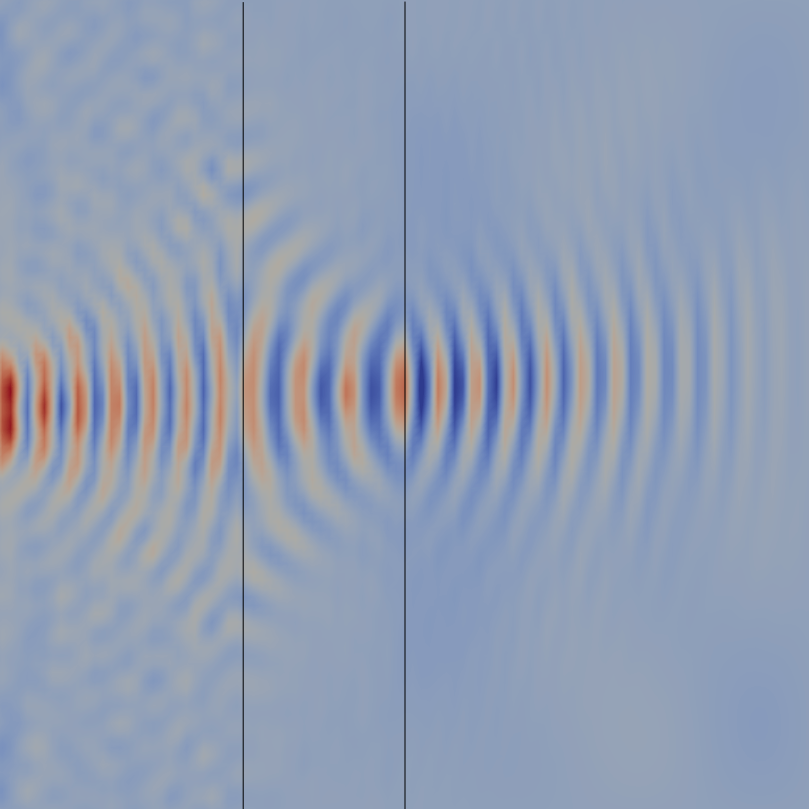} \\
\caption{Wave propagation with $p_D$ in \eqref{eq:bc-left-source} at $t = 0.06, 0.16, 0.36, 0.48$ }
\label{fig:wave-left-source}
\end{figure} 

We present the results of three experiments for $\mu_f = 18, 19, 20$. The finite elements with $\Vb_h(K) \times Q_h(K) \times \Wb_h(K) = {\rm RTN}_1 (K) \times \mc{P}_1 (K) \times \mc{P}_1(K; \R^d)$ are used for spatial discretization and $\mc{T}_h$ is the structured mesh obtained by bisecting $50 \times 50$ uniform squares of $\Omega$. The Crank--Nicolson scheme is used for time discretization with $\Delta t = 0.002$.

The wave propagation patterns are presented in Figure~\ref{fig:wave-mu-18}, Figure~\ref{fig:wave-mu-19}, and Figure~\ref{fig:wave-mu-20} for $\mu_f = 18$, $\mu_f = 19$, and $\mu_f = 20$, respectively. We call the three regions the left, the middle, and the right subdomains. In all of the figures in Figures~\ref{fig:wave-mu-18}--\ref{fig:wave-mu-20} the wave propagation patterns look standard plane waves in the lower part of the left subdomain whereas they are more complicated due to the waves reflected by the interface of the left and the middle subdomains. We can clearly see reversed wave propagation patterns on the metamaterial layer $\Omega_0$ in the three figures at $t=0.4$. In addition, wave propagation patterns on the right subdomain in the figures at $t=0.4$, are nearly plane waves with propagation directions similar to the patterns on the left subdomain. 
One can see that the details of wave propagation patterns, particularly the shapes and directions of the reversed patterns on the metamaterial layer, depend on the frequency of waves.

In the last experiment we set $\Omega = [0,2]\times [0,2]$ with $\Omega_0 = [3/5,1]\times [0,2]$, and the parameters are 
\algns{
\Omega_{\rho} = \Omega_{\kappa} = 
\begin{cases} 
80 \quad \text{ on }\Omega_0 \\
0 \quad \text{ on } \Omega \setminus \Omega_0 
\end{cases} , 
\qquad \omega_{\rho} = \omega_{\kappa} = 80 \quad \text{ on } \Omega. 
}
We also set 
\algn{ \label{eq:bc-left-source}
p_D(t,x,y) = 
\begin{cases}
	10 \exp( -(1+\sin (20 \pi (x^2+(y-1)^2- 10 t) ) )& \text{ if } y-1 < 0.1, \\
	0 & \text{ otherwise} .
\end{cases} 
}
We impose the boundary condition \eqref{eq:general-bc} with the above $p_D$ on the left-side $\{0\} \times [0,2]$ and with 0 on the other sides of $\Omega$. The finite elements are $\Vb_h(K) \times Q_h(K) \times \Wb_h(K) = {\rm RTN}_1 (K) \times \mc{P}_1 (K) \times \mc{P}_1(K; \R^d)$ and $\mc{T}_h$ is the structured mesh same as the second set of experiments. $\Delta t = 0.002$ with the Crank--Nicolson scheme. The wave propagation patterns are presented in Figure~\ref{fig:wave-left-source}. One can see that wave propagation patterns are not conventional on the metamaterial layer $\Omega_0$.

\section{Conclusion}
In this paper we developed finite element methods for acoustic wave propagation in the Drude-type metamaterials. We combined the mixed finite elements for the Poisson equations and piecewise discontinuous finite element spaces for spatial discretization. For time discretization we use the Crank--Nicolson scheme. We carried out the a priori error analysis and proposed a local post-processing scheme to overcome low approximation property of the pressure for the BDM type finite elements. The numerical experiments show the validity of our theoretical analysis as well as atypical wave propagation patterns in metamaterials.

\bibliography{reference}
\bibliographystyle{amsplain}
\vspace{.125in}

\end{document}

%% file: jlmacros.tex
\newcommand{\pmat}[1]{\begin{pmatrix} #1 \end{pmatrix}}

\newcommand{\beq} {\begin{equation}}
\newcommand{\eeq} {\end{equation}}
\newcommand{\bdm} {\begin{displaymath}}
\newcommand{\edm} {\end{displaymath}}

\newcommand{\bit}{\begin{itemize}}
\newcommand{\eit}{\end{itemize}}
\newcommand{\bde}{\begin{description}}
\newcommand{\ede}{\end{description}}

\newcommand{\ben}{\begin{enumerate}}
\newcommand{\een}{\end{enumerate}}
\newcommand{\algn}[1]{\begin{align} #1 \end{align}}
\newcommand{\algns}[1]{\begin{align*} #1 \end{align*}}
\newcommand{\mltln}[1]{\begin{multline} #1 \end{multline}}
\newcommand{\mltlns}[1]{\begin{multline*} #1 \end{multline*}}

\newcommand{\gats}[1]{\begin{gather*} #1 \end{gather*}}
\newcommand{\barr}{\begin{array}}
\newcommand{\earr}{\end{array}}

\newcommand{\mc}[1]{\mathcal{#1}}
\newcommand{\LRp}[1]{\left( #1 \right)}

\newcommand{\ra}{\rightarrow}

\newcommand{\R}{{\mathbb R}}

\newcommand{\grad}{\operatorname{grad}}

\renewcommand{\div}{\operatorname{div}}

\newcommand{\esssup}{\operatorname{esssup}}

\newcommand{\bs}{\boldsymbol}

\newcommand{\pd}{\partial}

\newtheorem{theorem}{Theorem}[section]
\newtheorem{prop}[theorem]{Proposition} 
\newtheorem{lemma}[theorem]{Lemma}
